\documentclass[11pt]{amsart}

\usepackage[english]{babel}

\usepackage[letterpaper,top=2cm,bottom=2cm,left=2cm,right=2cm,marginparwidth=1.75cm]{geometry}

\usepackage{mathtools,amssymb,amsthm,bbold}
\usepackage{tikz-cd, tikz}
\usepackage[mathscr]{euscript}
\usepackage{graphicx}
\usepackage[colorlinks=true, allcolors=blue]{hyperref}

\usepackage{setspace}
\tikzset{
    labl/.style={anchor=south, rotate=90, inner sep=.5mm}
}

\theoremstyle{definition}
\newtheorem{prop}{Proposition}[section]
\newtheorem{lem}[prop]{Lemma}
\newtheorem{thm}[prop]{Theorem}
\newtheorem{cor}[prop]{Corollary}
\newtheorem{defn}[prop]{Definition}
\newtheorem{conj}[prop]{Conjecture}
\newtheorem{rem}[prop]{Remark}

\title{A Quiver Analogue of Higman's Conjecture}
\author{Lucien Hennecart and Nikolai Perry}
\date{}

\address{School of Mathematics, University of Edinburgh, Edinburgh, UK}
\email{lucien.hennecart@ed.ac.uk}
\email{N.Perry-1@sms.ed.ac.uk}

\begin{document}
\maketitle

\begin{abstract}
An unresolved conjecture by Graham Higman states that for all $n\geq 1$ the number of conjugacy classes of the group of $n \times n$ unitriangular matrices with entries in the finite field $\mathbb{F}_q$ is a polynomial in $q$. In this paper we introduce a new quiver generalization of the conjecture. Motivated by this generalization, we prove that certain operations on quivers leave the relevant counts unchanged. Based on these invariance properties, we solve the introduced conjecture for quivers containing no path of length exceeding two, providing explicit formulas.
\end{abstract}

\renewcommand{\abstractname}{Acknowledgements}
\begin{abstract}
Nikolai Perry is grateful for the support of this work by the Edinburgh Mathematical Society through their student research bursary, as well as the University of Edinburgh's college vacation scholarship.
\end{abstract}

\section{Introduction}
We denote by $k(U_n(\mathbb{k}))$ the number of conjugacy classes of the group $U_n(\mathbb{k})$ of $n \times n$ unitriangular matrices with entries in $\mathbb{k} = \mathbb{F}_q$. First mentioned in \cite{Higman}, Higman's conjecture asserts that $k(U_n(\mathbb{k}))$ is a polynomial in $q$, for all $n \geq 1$. It has been verified for $n\leq 13$ by Arregi and Vera-L\'{o}pez in a series of papers \cite{VA1, VA2, VA3} where a clever brute-force algorithm was devised. The authors have shown that each conjugacy class of $U_n$ contains a unique \textit{canonical} matrix with respect to a certain order on matrix entries, and that a matrix is canonical if and only if certain entries called \textit{inert points} are zero. Their algorithm revolves around constructing canonical matrices while identifying which entries are inert. 

Pak and Soffer \cite{PS} have built on this algorithm while developing a recursive way of calculating the number of conjugacy classes of pattern groups (subgroups of $U_n(\mathbb{k})$ where certain entries are fixed to be zero), verifying the conjecture for $n\leq16$. They have also found a pattern subgroup of $U_{13}(\mathbb{k})$ whose number of conjugacy classes is not polynomial count, and which appears in the recursive expansion for $k(U_{59}(\mathbb{k}))$. This is the main indication that Higman's conjecture could be false.

Mozgovoy \cite{Mozgovoy} has recently given a quiver generalization of the conjecture, as well as a modern way to understand it in terms of linear stacks. This generalization involves the number of commuting pairs of nilpotent endomorphisms of a projective representation of an acyclic quiver. We propose another generalization by instead considering the count of commuting pairs of radical endomorphisms. This count is expected to be related to the character of the Hall algebra that can be constructed from the category of pairs $(P,f)$ of a projective representation of an acyclic quiver together with a radical endomorphism of it. This is a sign that Hall algebra techniques could be fruitful to understand Higman's conjecture or, in the reverse direction, that polynomiality of the character of the Hall algebra witnesses deeper algebraic structures (as for the constructible Hall algebra \cite{sevenhant2001relation,bozec2019counting}, where a $\mathbb{Z}$-graded Borcherds Lie algebra appears). In this paper, we concentrate ourselves on counting properties, leaving the Hall algebra part for further investigations.

After covering some background material and, in particular, the concept of radicals in \S\ref{preliminaries}, we introduce the new, natural quiver generalization in \S\ref{a quiver generalization}, proving results that let us reduce the associated counting problems. We give a complete description of the first nontrivial cases in \S\ref{examples} and explain future research directions in \S\ref{furtherdirections}.

\section{Preliminaries}\label{preliminaries}

Burnside's Lemma allows us to write

\begin{equation*}
    k(U_n(\mathbb{k})) = \frac{\left| \left\{(x,y) \in U_n(\mathbb{k}) \times U_n(\mathbb{k}) \; \middle| \; xy - yx = 0 \right\}\right|}{\left| U_n(\mathbb{k}) \right|},
\end{equation*}
so we see that $k(U_n(\mathbb{k}))$ is a polynomial in $q$ if and only if the number of commuting pairs in $U_n(\mathbb{k})$ is a polynomial in $q$. The elements $x,y \in U_n(\mathbb{k})$ commute precisely when $x-\mathbb{1}_{n \times n}, y-\mathbb{1}_{n \times n} \in T_n$ commute, where we have defined $T_n = T_n(\mathbb{k}) := U_n(\mathbb{k}) - \mathbb{1}_{n \times n}$.
Therefore Higman's conjecture is equivalent to the assertion that 
\begin{equation}\label{classical conjecture, commuting}
    \left| \left\{(x,y) \in T_n \times T_n \; \middle| \; xy - yx = 0 \right\}\right| \in \mathbb{Q}[q], \; \text{for all $n \geq 1$}.
\end{equation}

To identify the object $T_n$, we recall the definition of the radical of a ring:
\begin{defn}[Radical of a Ring]\label{radical of a ring}
The (Jacobson) \textit{radical} $\operatorname{rad}R$ of a ring $R$ is the intersection of all maximal right ideals in $R$.
\end{defn}
\begin{lem}[{{\cite[Lemma 4.1]{Schiffler}}}]\label{radical lemma}
Let $R$ be a ring and $a \in R$. The following are equivalent:
\begin{enumerate}
    \item $a \in \operatorname{rad}R$.
    \item $a$ lies in the intersection of all maximal left ideals in $R$.
    \item For all $b \in R$, the element $1-a\cdot_R b$ is invertible in $R$.
    \item For all $b \in R$, the element $1-b\cdot_R a$ is invertible in $R$.
\end{enumerate}
\end{lem}
\begin{cor}
For any ring $R$, $\operatorname{rad}R$ is a two-sided ideal in $R$.
\end{cor}

It is the case that $T_n = \operatorname{rad}B_n$, where $B_n = B_n(\mathbb{k})$ is the algebra of upper triangular $n \times n$ matrices. Since $T_n$ is a nilpotent ideal, we have a natural way of weakening Higman's conjecture as given in the form of (\ref{classical conjecture, commuting}) by considering powers of the radical: we can ask if, given $l,m \geq 0$ and $n\geq 1$, we have that 
\begin{equation}\label{weakening}
    \left| \left\{(x,y) \in T_n^l \times T_n^l \; \middle| \; xy - yx \in T_n^m \right\}\right| \in \mathbb{Q}[q].
\end{equation}
Note that this is trivially the case when $m \leq 2l$, and that Higman's conjecture is the statement that the answer is ``yes'' for all $n\geq 1$ with $l=1$ and $m=n$.

We also have a natural way of generalizing Higman's conjecture --- we can investigate counting problems relating to radicals of certain algebras that arise from quivers. For this it helps to recall the definition of the radical of a category, for we will be interested in the category of representations of an arbitrary acyclic quiver $Q$, $\operatorname{rep}Q$:
\begin{defn}[Radical of a Category]
The (Jacobson) \textit{radical} $\operatorname{Rad} = \operatorname{Rad}_{\mathcal{C}}$ of an additive $\mathbb{k}$-category $\mathcal{C}$ is the class of morphisms defined as, for all objects $X,Y \in \mathcal{C}$, 
\begin{equation*}\label{radical of a category}
    \operatorname{Rad}(X,Y) = \{f \in \operatorname{Hom}(X,Y) \; | \; 1_{X} - g\circ f \; \text{is an isomorphism $\forall g \in \operatorname{Hom}(Y,X)$}\}
    \subseteq \operatorname{Hom}(X,Y).
\end{equation*}
\end{defn}
Notice that Definitions \ref{radical of a ring} and \ref{radical of a category} align when we consider the endomorphism algebra $\operatorname{End}(X)$ of an object $X \in \mathcal{C}$:  $\operatorname{rad}(\operatorname{End}(X)) = \operatorname{Rad}(X,X)$.

\section{A Quiver Generalization} \label{a quiver generalization}

\subsection{The Conjecture}

The classification of projective representations of quivers does not depend on the ground field $\mathbb{k}$. Using this, Remark \ref{rigorous sense} will make rigorous sense of our following generalization of Higman's conjecture:

\begin{conj}\label{THE CONJECTURE}
Let $Q$ be an acyclic quiver and $P \in \operatorname{rep}Q$ be a projective representation of $Q$ over the field $\mathbb{k} = \mathbb{F}_q$. Then the number of commuting pairs in the radical $\operatorname{rad}(\operatorname{End}(P))$ is a polynomial in $q$.
\end{conj}

\begin{defn}
\label{dfn:projectiverepresentation}
Let $\bigoplus_{i \in Q_0} P_Q(i)^{d_i}$ be a projective representation of an acyclic quiver $Q$, where $P_Q(i)$ denotes the indecomposable projective representation at vertex $i$, and $\mathbf{d} \in \mathbb{Z}_{\geq 0}^{Q_0}$ is what we call the \textit{summand vector}. We denote this representation by $P_{Q, \mathbf{d}}$ and define $A_{Q, \mathbf{d}}$ to be its endomorphism algebra, $\operatorname{End}(P_{Q, \mathbf{d}})$. We define the \textit{count} $[Q,\mathbf{d}]$ to be the number of commuting pairs in $\operatorname{rad}A_{Q, \mathbf{d}} \times \operatorname{rad}A_{Q, \mathbf{d}}$.
\end{defn}

\begin{rem}
Note that $[Q, \mathbf{d}] = 1$ in the degenerate cases where $\mathbf{d} = \mathbf{0}$ or $Q_0 = \varnothing$, since an empty direct sum is the zero object.
\end{rem}

Suppose that we have a quiver
$Q=
\begin{tikzpicture}[baseline={([yshift=-.7ex]current bounding box.center)}]
\draw (0,0) node{1};
\draw (1.5,0) node{2};
\draw (3,0) node {3};
\draw[->] (0.2,0) -- (1.3,0);
\draw[->] (2.8,0) -- (1.7,0);
\end{tikzpicture}
$
and summand vector $\mathbf{d} = (4,0,1)$. Then we compactly visualise this data as
$\left[
\begin{tikzpicture}[baseline={([yshift=-.7ex]current bounding box.center)}]
\draw (0,0) node{4};
\draw (1.5,0) node{0};
\draw (3,0) node {1};
\draw[->] (0.2,0) -- (1.3,0);
\draw[->] (2.8,0) -- (1.7,0);
\end{tikzpicture}
\right],$
which we also take to be alternative notation for $[Q,\mathbf{d}]$. Note that when in square brackets each vertex $i \in Q_0$ is labeled by $d_i$ instead of $i$.

\begin{rem}\label{rigorous sense}
Since every projective representation of $Q$ is isomorphic to $P_{Q, \mathbf{d}}$ for some summand vector $\mathbf{d}$ \cite[Corollary 2.21]{Schiffler}, Conjecture \ref{THE CONJECTURE} is equivalent to the statement that for any acyclic quiver $Q$ and summand vector $\mathbf{d}$, the count $[Q, \mathbf{d}]$ is a polynomial in $q$.
\end{rem}

This is subtly --- though crucially --- different to that of Mozgovoy's generalization \cite[Conjecture 1]{Mozgovoy}, in which it is the set of nilpotent elements that is considered, not the radical. Our generalization is natural, since the radical of a finite dimensional algebra is always a nilpotent \textit{ideal} \cite[Corollary 4.10]{Schiffler}, allowing for a weakening as in (\ref{weakening}). The two generalizations agree precisely when the set of nilpotent elements is a two-sided ideal, as is the case for quivers and summand vectors of the form
\begin{equation*}
    \left[
    \underbrace{
    \begin{tikzpicture}[baseline={([yshift=-.7ex]current bounding box.center)}]
    \draw (0,0) node{1};
    \draw[->](0.2,0) -- (1.3,0);
    \draw (1.5,0) node{1};
    \draw[->](1.7,0) -- (2.8,0);
    \draw (3.3,0) node{$\ldots$};
    \draw[->](3.8,0) -- (4.9,0);
    \draw (5.1,0) node{1};
    \end{tikzpicture}
    }_{n}
    \right],
\end{equation*}
which recovers Higman's classical conjecture (\ref{classical conjecture, commuting}) as the endomorphism algebra is isomorphic to $B_n$.

Regardless of whether Conjecture \ref{THE CONJECTURE} is true, it serves as a good motivation to study the combinatorics of certain objects arising from the endomorphism algebras of projective representations of quivers, in our case --- counts. To this end, we will prove several results that allow us to understand the nature of counts of acyclic quivers and summand vectors in terms of counts of (usually) simpler quivers and summand vectors.

In \S \ref{reductions} we prove that certain operations on quivers and summand vectors leave the algebras unchanged up to isomorphism, hence preserving counts. In \S \ref{more reductions} we present operations that preserve the structures of the radicals and thereby preserve counts, though not necessarily preserving the structures of the algebras.

\subsection{Notations}

Suppose that the object $X$ in an additive $\mathbb{k}$-category $\mathcal{C}$ is given by the (external) direct sum $X = \bigoplus_{i=1}^t X_i$. Now consider the endomorphism algebra $\operatorname{End}(X)$.
Letting $\pi_k:X\longrightarrow X_k$ and $\iota_k:X_k \longrightarrow X$ denote the canonical projection and inclusion morphisms, any $f \in \operatorname{End}(X)$ is uniquely determined by its matrix elements $f_{ij} := \pi_i \circ f \circ \iota_j \in \operatorname{Hom}(X_j, X_i)$ for $i,j=1,\dots,t$.
Furthermore, we have that $(f \circ g)_{ij} = \sum_{k=1}^t f_{ik} \circ g_{kj}$, and so the identification
\begin{equation*}
    f \xleftrightarrow{} 
    \begin{bmatrix}
        f_{11} & \cdots & f_{1t}\\
        \vdots & \ddots & \vdots\\
        f_{t1} & \cdots & f_{tt}
    \end{bmatrix}
    = [f_{ij}]_{i,j = 1,\dots,t}
\end{equation*}
yields an algebra isomorphism
\begin{equation*}
    \operatorname{End}(X) \cong
    \begin{bmatrix}
        \operatorname{Hom}(X_1, X_1) & \cdots & \operatorname{Hom}(X_t, X_1)\\
        \vdots & \ddots & \vdots\\
        \operatorname{Hom}(X_1, X_t) & \cdots & \operatorname{Hom}(X_t, X_t)
    \end{bmatrix}
    = [\operatorname{Hom}(X_j, X_i)]_{i,j = 1,\dots,t},
\end{equation*}
where addition and scalar multiplication in the right-hand algebra is component-wise and multiplication is as above: $(f \circ g)_{ij} = \sum_{k=1}^t f_{ik} \circ g_{kj}$. We will often abuse notation by writing 
$$\operatorname{End}(X) = [\operatorname{Hom}(X_j, X_i)]_{i,j = 1,\dots,t} \;\; \text{and} \;\; f = [f_{ij}]_{i,j = 1,\dots,t} \in \operatorname{End}(X).$$
Similarly, we will understand $\operatorname{rad}(\operatorname{End(X)})$ as $[\operatorname{Rad}(X_j, X_i)]_{i,j = 1,\dots,t}$ \cite[A.3, 3.4. Lemma]{assem_skowronski_simson_2006}.

Given an acyclic quiver $Q$ and summand vector $\mathbf{d}$, we often write $P_{Q, \mathbf{d}}$ as $\bigoplus_{i = 1}^t P_Q(v_i)$, where $v_1, \dots, v_t$ are (not necessarily distinct) vertices in $Q$, so that
$$A_{Q, \mathbf{d}} = \operatorname{End}(P_{Q, \mathbf{d}}) \\= [\operatorname{Hom}(P_Q(v_j), P_Q(v_i))]_{i,j = 1,\dots,t} \;\; \text{and} \;\; \operatorname{rad}A_{Q, \mathbf{d}} = [\operatorname{Rad}(P_Q(v_j), P_Q(v_i))]_{i,j = 1,\dots,t}.
$$ We will also often abbreviate $\operatorname{Hom}(P_Q(v_j), P_Q(v_i))$ by $\operatorname{H}_{ij}$ and $\operatorname{Rad}(P_Q(v_j), P_Q(v_i))$ by $\operatorname{R}_{ij}$.

\begin{rem}\label{map_def}
If for each $i,j$ we define the maps
\begin{align*}
    \phi_{ij}: \operatorname{H}_{ij} = \operatorname{Hom}(P_Q(v_j), P_Q(v_i)) &\longrightarrow P_Q(v_i)_{v_j} = \langle c \; | \; c:v_i \curvearrowright v_j \in Q \rangle; \\
    h = (h^{(k)})_{k \in Q_0} &\longmapsto h^{(v_j)}(e_{v_j}),
\end{align*}
where $\langle c \; | \; c:v_i \curvearrowright v_j \in Q \rangle$ is the $\mathbb{k}$-vector space with basis the set of all paths $v_i \curvearrowright v_j$ in $Q$, then these are vector space isomorphisms \cite[Theorem 2.11]{Schiffler}, and so
$\{\phi_{ij}^{-1}(c) \; | \; c:v_i \curvearrowright v_j \in Q\}$
is a basis for $\operatorname{H}_{ij}$.
\end{rem}

\begin{rem}\label{rad_map_def}
Recall Definition \ref{radical of a category} of the radical of a category. Consider $\operatorname{H}_{ij}$ for $i,j=1,\dots,t$. If $v_i \neq v_j$, then we have the implication $\operatorname{H}_{ij} \neq 0 \implies \operatorname{H}_{ji} = 0$, since Q is acyclic. Therefore we see that $\operatorname{R}_{ij} = \operatorname{H}_{ij}$, if $v_i \neq v_j$. On the other hand, if $v_i =v_j$, then $\operatorname{H}_{ij} \cong \langle e_{v_i} \rangle \cong \mathbb{k}$, and so $\operatorname{R}_{ij} = 0$. Thus, for any $i,j$, it is the case that the isomorphism $\phi_{ij}$ in Remark \ref{map_def} restricted to $\operatorname{R}_{ij}$ (which we denote by $\psi_{ij}$) yields a vector space isomorphism 
\begin{align*}
    \psi_{ij}: \operatorname{R}_{ij} = \operatorname{Rad}(P_Q(v_j), P_Q(v_i)) 
    \; \stackrel{\sim}\longrightarrow \; \langle c \; | \; c:v_i \stackrel{\text{n.c.}}\curvearrowright v_j \in Q \rangle,
\end{align*}
where $\langle c \; | \; c:v_i \stackrel{\text{n.c.}}\curvearrowright v_j \in Q \rangle$ is the $\mathbb{k}$-vector space with basis the set of all non-constant paths $v_i \stackrel{\text{n.c.}}\curvearrowright v_j$ in $Q$. In particular, $\{\psi_{ij}^{-1}(c) \; | \; c:v_i \stackrel{\text{n.c.}}\curvearrowright v_j \in Q\}$
is a basis for $\operatorname{R}_{ij}$.
\end{rem}

\begin{prop}[Composition of Basis Elements]\label{basis_composition}
$ $\newline
\begin{enumerate}
    \item 
    Let $c:v_i \curvearrowright v_k \in Q$ and $p:v_k \curvearrowright v_j \in Q$ be paths. Then
    \begin{equation*}
        \phi_{ik}^{-1}(c) \circ \phi_{kj}^{-1}(p) = \phi_{ij}^{-1}(cp).
    \end{equation*}
    
    \item 
    Let $c:v_i \stackrel{\text{n.c.}}\curvearrowright v_k \in Q$ and $p:v_k \stackrel{\text{n.c.}}\curvearrowright v_j \in Q$ be non-constant paths. Then
    \begin{equation*}
        \psi_{ik}^{-1}(c) \circ \psi_{kj}^{-1}(p) = \psi_{ij}^{-1}(cp).
    \end{equation*}
\end{enumerate}
\end{prop}
\begin{proof}
$ $\newline
\begin{enumerate}
    \item 
    
    It suffices to show that $\phi_{ij}(\phi_{ik}^{-1}(c) \circ \phi_{kj}^{-1}(p)) = cp$. Indeed,
    \begin{align*}
        \phi_{ij}(\phi_{ik}^{-1}(c) \circ \phi_{kj}^{-1}(p)) 
        &= (\phi_{ik}^{-1}(c) \circ \phi_{kj}^{-1}(p))^{(v_j)}(e_{v_j}) 
        = (\phi_{ik}^{-1}(c))^{(v_j)} \circ (\phi_{kj}^{-1}(p)))^{(v_j)}(e_{v_j})\\
        &= (\phi_{ik}^{-1}(c))^{(v_j)} \circ \phi_{kj}(\phi_{kj}^{-1}(p)) 
        = (\phi_{ik}^{-1}(c))^{(v_j)}(p) 
        = (\phi_{ik}^{-1}(c))^{(v_j)} \circ \varphi_{p}(e_{v_k})\\
        &\stackrel{(\star)}{=} \varphi_{p} \circ (\phi_{ik}^{-1}(c))^{(v_k)} (e_{v_k})
        = \varphi_{p} \circ \phi_{ik}(\phi_{ik}^{-1}(c))
        = \varphi_{p}(c)
        = cp,
    \end{align*}
    where $\varphi_p$ is the linear mapping given by the action of $p$ and in $(\star)$ we made use of the following commutative diagram:

    \begin{center}
    \begin{tikzcd}[row sep=10ex, column sep=10ex]
    P_Q(v_k)_{v_k} \arrow[r, "\varphi_p", bend left=10] \arrow[d, "(\phi_{ik}^{-1}(c))^{(v_k)}", swap] 
    & P_Q(v_k)_{v_j} \arrow[d, "(\phi_{ik}^{-1}(c))^{(v_j)}"] \\
    P_Q(v_i)_{v_k} \arrow[r, "\varphi_p", bend left=10] 
    & P_Q(v_i)_{v_j}
    \end{tikzcd}
    \end{center}
    
    \item This follows from the fact that each $\psi_{ij}$ is a restriction of $\phi_{ij}$.
\end{enumerate}
\end{proof}

\subsection{Count-Preserving Operations}\label{reductions}
The following theorem and corollary show that counts are invariant under reversal of all arrows in their corresponding quivers. This is useful for later, where we will prove that certain operations on sources leave counts unchanged, and thus it will follow that the analogous operations on sinks leave counts unchanged too.

\begin{thm}\label{opposite_quiver}
Let $Q$ be an acyclic quiver and $\mathbf{d} \in \mathbb{Z}_{\geq 0}^{Q_0}$. Then $A^{\operatorname{op}}_{Q, \mathbf{d}} \cong A_{Q^{\operatorname{op}}, \mathbf{d}}$
, as $\mathbb{k}$-algebras.
\end{thm}

\begin{proof}

We write $P_{Q, \mathbf{d}}$ as $\bigoplus_{i = 1}^t P_Q(v_i)$ and $P_{Q^{\operatorname{op}}, \mathbf{d}}$ as $\bigoplus_{i = 1}^t P_{Q^{\operatorname{op}}}(v_i)$ so that
\begin{align*}
    A_{Q,\mathbf{d}} 
    &= [\operatorname{Hom}(P_Q(v_j), P_Q(v_i))]_{i,j = 1,\dots,t} 
    = [\operatorname{H}_{ij}]_{i,j = 1,\dots,t} \quad \text{and}\\
    A_{Q^{\operatorname{op}},\mathbf{d}} 
    &= [\operatorname{Hom}(P_{Q^{\operatorname{op}}}(v_j), P_{Q^{\operatorname{op}}}(v_i))]_{i,j = 1,\dots,t} 
    = [\operatorname{H}_{ij}^{\operatorname{op}}]_{i,j = 1,\dots,t}.
\end{align*}
Recall that the opposite algebra $A_{Q,\mathbf{d}}^{\operatorname{op}}$ is the same as $A_{Q,\mathbf{d}}$ but with multiplication reversed, i.e. \\$f \cdot_{A_{Q,\mathbf{d}}^{\operatorname{op}}}g = g \cdot_{A_{Q,\mathbf{d}}}f$.\\

For each $i,j = 1,\dots, t$ we define the vector space isomorphisms
\begin{align*}
    \phi_{ij}: \operatorname{H}_{ij} \; \stackrel{\sim}\longrightarrow \; \langle c \; | \; c:v_i \curvearrowright v_j \in Q \rangle \quad \text{and} \quad
    \phi_{ij}^{\operatorname{op}}: \operatorname{H}_{ij}^{\operatorname{op}} \; \stackrel{\sim}\longrightarrow \; \langle c \; | \; c:v_i \curvearrowright v_j \in Q^{\operatorname{op}} \rangle
\end{align*}
as in Remark \ref{map_def}. Now $c$ is a path $v_i \curvearrowright v_j$ in $Q$ if and only if $c^{\operatorname{op}}$ is a path $v_j \curvearrowright v_i$ in $Q^{\operatorname{op}}$, so for each $i,j = 1,\dots,t$ we can define vector space isomorphisms by specifying their actions on the basis elements:
\begin{equation*}
    \Phi_{ij}: \operatorname{H}_{ij} \; \stackrel{\sim}\longrightarrow \; \operatorname{H}_{ji}^{\operatorname{op}}; \quad \phi_{ij}^{-1}(c) \longmapsto (\phi_{ji}^{\operatorname{op}})^{-1}(c^{\operatorname{op}}).
\end{equation*}

Now note that if $c:v_j \curvearrowright v_k$ and $p:v_k \curvearrowright v_i$ are paths in $Q$, then
\begin{align*}
    \Phi_{ji}(\phi_{jk}^{-1}(c) \circ \phi_{ki}^{-1}(p))
    &= \Phi_{ji}(\phi_{ji}^{-1}(cp)) 
    = (\phi_{ij}^{\operatorname{op}})^{-1}((cp)^{\operatorname{op}})
    = (\phi_{ij}^{\operatorname{op}})^{-1}(p^{\operatorname{op}} c^{\operatorname{op}})
    = (\phi_{ik}^{\operatorname{op}})^{-1}(p^{\operatorname{op}}) \circ (\phi_{kj}^{\operatorname{op}})^{-1}(c^{\operatorname{op}})\\
    &= \Phi_{ki}(\phi_{ki}^{-1}(p)) \circ \Phi_{jk}(\phi_{jk}^{-1}(c)),
\end{align*}
where we have used Proposition \ref{basis_composition} twice. And so it follows from the linearity of each $\Phi_{ij}$ that for any $g_{jk} \in \operatorname{H}_{jk}$ and $f_{ki} \in \operatorname{H}_{ki}$
we have 
\begin{equation}\label{mult_preservation}
    \Phi_{ji}(g_{jk} \circ f_{ki}) = \Phi_{ki}(f_{ki}) \circ \Phi_{jk}(g_{jk}).
\end{equation}

Finally, we define the map $\mathbf{\Phi}: A_{Q,\mathbf{d}}^{\operatorname{op}} \; \longrightarrow A_{Q^{\operatorname{op}}, \mathbf{d}}$ by $(\mathbf{\Phi}(f))_{ij} = \Phi_{ji}(f_{ji})$. We show that this is an isomorphism of $\mathbb{k}$-algebras, hence proving the theorem:
\begin{itemize}
    \item 
    This is an isomorphism of vector spaces since it can be viewed as component-wise application of the vector space isomorphisms $\Phi_{ij}$ followed by transposition.
    
    \item
    The unity $1_{A_{Q,\mathbf{d}}^{\operatorname{op}}}$ of $A_{Q,\mathbf{d}}^{\operatorname{op}}$ is such that
    \begin{equation*}
        (1_{A_{Q,\mathbf{d}}^{\operatorname{op}}})_{ij}
        = 
        \begin{cases}
            0, \;\; \text{if} \;\; i \neq j, \\ 1_{\operatorname{End}(P_Q(v_i))}, \;\; \text{if} \;\; i = j,
        \end{cases}
    \end{equation*}
    and so
    \begin{equation*}
        (\mathbf{\Phi}(1_{A_{Q,\mathbf{d}}^{\operatorname{op}}}))_{ij} 
        = \Phi_{ji}((1_{A_{Q,\mathbf{d}}^{\operatorname{op}}})_{ji})
        =
         \begin{cases}
            0, \;\; \text{if} \;\; i \neq j, \\ \Phi_{ii}(1_{\operatorname{End}(P_Q(v_i))}), \;\; \text{if} \;\; i = j.
        \end{cases}
    \end{equation*}
    But $\Phi_{ii}(1_{\operatorname{End}(P_Q(v_i))}) = \Phi_{ii}(\phi_{ii}^{-1}(e_{v_i})) = (\phi_{ii}^{\operatorname{op}})^{-1}(e_{v_i}) = 1_{\operatorname{End}(P_{Q^{\operatorname{op}}}(v_i))}$, so we see that \\
    $\mathbf{\Phi}(1_{A_{Q,\mathbf{d}}^{\operatorname{op}}}) = 1_{A_{Q^{\operatorname{op}},\mathbf{d}}}$.
    
    \item
    Let $f,g \in A_{Q, \mathbf{d}}^{\operatorname{op}}$. We have that
    \begin{align*}
        (\mathbf{\Phi}(f \cdot_{A_{Q, \mathbf{d}}^{\operatorname{op}}} g))_{ij} 
        &= \Phi_{ji}((f \cdot_{A_{Q, \mathbf{d}}^{\operatorname{op}}} g)_{ji})
        = \Phi_{ji}((g \cdot_{A_{Q, \mathbf{d}}} f)_{ji})
        = \sum_{k} \Phi_{ji}(g_{jk} \circ f_{ki})\\
        &\stackrel{(\ref{mult_preservation})}= \sum_{k} \Phi_{ki}(f_{ki}) \circ \Phi_{jk}(g_{jk})
        = \sum_{k} (\mathbf{\Phi}(f))_{ik} \circ (\mathbf{\Phi}(g))_{kj},
    \end{align*}
    so 
    $\mathbf{\Phi}(f \cdot_{A_{Q, \mathbf{d}}^{\operatorname{op}}} g)
    = \mathbf{\Phi}(f) \cdot_{A_{Q^{\operatorname{op}}, \mathbf{d}}} \mathbf{\Phi}(g)$.
\end{itemize}
Thus $\mathbf{\Phi}$ is an isomorphism of $\mathbb{k}$-algebras.
\end{proof}

\begin{cor}[Arrow Reversal]
\label{cor:opposite}
If $Q$ is acyclic and $\mathbf{d} \in \mathbb{Z}_{\geq 0}^{Q_0}$, then $[Q, \mathbf{d}] = [Q^{\operatorname{op}}, \mathbf{d}]$.
\end{cor}
\begin{proof}
For any algebra $A$, by Lemma \ref{radical lemma} we have that
\begin{align*}
    x \in \operatorname{rad}A \; &\Leftrightarrow \; \forall y \in A, 1 - x \cdot_A y \;\; \text{is invertible in} \; A\\
    &\Leftrightarrow \; \forall y \in A^{\operatorname{op}}, 1 - y \cdot_{A^{\operatorname{op}}} x \;\; \text{is invertible in} \; A^{\operatorname{op}} \;
    \Leftrightarrow \; x \in \operatorname{rad}A^{\operatorname{op}},
\end{align*}
so in particular $\operatorname{rad}A$ and $\operatorname{rad}A^{\operatorname{op}}$ have the same number of commuting pairs. The result follows from the fact that $A_{Q, \mathbf{d}}^{\operatorname{op}} \cong A_{Q^{\operatorname{op}}, \mathbf{d}}$, by Theorem \ref{opposite_quiver}.
\end{proof}

Our next theorem shows that if $d_v = 0$ for some vertex $v$ in $Q$, then we can remove that vertex so long as we preserve the paths through it as follows:
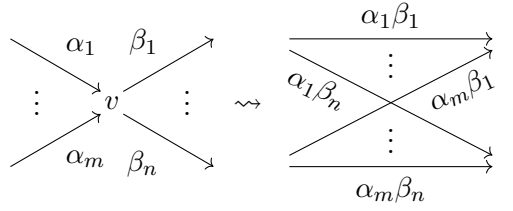
\begin{equation*}
    \begin{tikzpicture}[baseline={([yshift=-.7ex]current bounding box.center)}]
    \draw (1.5,1) node{$v$};
    \draw (0,0);
    \draw (0,2);
    \draw (3,2);
    \draw (3,0);
    \draw (0.5,1.1) node{\vdots};
    \draw (2.5,1.1) node{\vdots};
    \draw[->] (0.15,0.15) -- (1.35,0.85) node[midway,below right] {$\alpha_m$};
    \draw[->] (0.15,1.85) -- (1.35,1.15) node[midway,above right] {$\alpha_1$};
    \draw[->] (1.65,1.15) -- (2.85,1.85) node[midway,above left] {$\beta_1$};
    \draw[->] (1.65,0.85) -- (2.85,0.15) node[midway,below left] {$\beta_n$};
    \end{tikzpicture}
    \rightsquigarrow
    \begin{tikzpicture}[baseline={([yshift=-.7ex]current bounding box.center)}]
    \draw (0,0);
    \draw (0,2);
    \draw (3,2);
    \draw (3,0);
    \draw (1.5,1.6) node{\vdots};
    \draw (1.5,0.6) node{\vdots};
    
    \draw[->] (0.15,0.15) -- (2.85,0.15) node[midway,below] {$\alpha_m \beta_n$};;
    \draw[->] (0.15,0.3) -- (2.85,1.7) node[below left, rotate = 25] {$\alpha_m \beta_1$};
    \draw[->] (0.15,1.7) node[below right, rotate = -25] {$\alpha_1 \beta_n$} -- (2.85,0.3);
    \draw[->] (0.15,1.85) -- (2.85,1.85) node[midway,above] {$\alpha_1 \beta_1$};
    \end{tikzpicture}.
\end{equation*}
\begin{thm}[Zero-Vertex Removal]\label{zero node removal}
Suppose that $Q$ is acyclic and $\mathbf{d} \in \mathbb{Z}_{\geq 0}^{Q_0}$ is such that $d_v=0$ for some vertex $v$. Let $\overline{Q}_0 := Q_0 \setminus \{v\}$ and $\overline{\mathbf{d}} := (d_i)_{i \in \overline{Q}_0}$. Further, let $\mathcal{A} := \{\alpha \in Q_1 \; | \; t(\alpha) = v\}$ and $\mathcal{B} := \{\beta \in Q_1 \; | \; s(\beta) = v\}$, so that $\mathcal{AB} = \{\alpha \beta \; | \; \alpha \in \mathcal{A}, \beta \in \mathcal{B}\}$ is the set of paths of length two in $Q$ with middle vertex $v$. Thinking of $\mathcal{AB}$ as a set of arrows, let $\overline{Q}_1 := Q_1 \cup \mathcal{AB} \setminus (\mathcal{A} \cup \mathcal{B})$, giving the quiver $\overline{Q} := (\overline{Q}_0, \overline{Q}_1)$. Then $A_{Q, \mathbf{d}} \cong A_{\overline{Q}, \overline{\mathbf{d}}}$ as $\mathbb{k}$-algebras. In particular, $[Q, \mathbf{d}] = [\overline{Q}, \overline{\mathbf{d}}]$.
\end{thm}

\begin{proof}
As before, we write $P_{Q, \mathbf{d}}$ as $\bigoplus_{i = 1}^t P_Q(v_i)$. Note that $v_i \neq v$ for all $i$, since $d_v = 0$, so we can consistently write $P_{\overline{Q}, \overline{\mathbf{d}}}$ as $\bigoplus_{i = 1}^t P_{\overline{Q}}(v_i)$. Hence
\begin{align*}
    A_{Q,\mathbf{d}} 
    &= [\operatorname{Hom}(P_Q(v_j), P_Q(v_i))]_{i,j = 1,\dots,t} 
    = [\operatorname{H}_{ij}]_{i,j = 1,\dots,t} \quad \text{and}\\
    A_{\overline{Q},\overline{\mathbf{d}}} 
    &= [\operatorname{Hom}(P_{\overline{Q}}(v_j), P_{\overline{Q}}(v_i))]_{i,j = 1,\dots,t} 
    = [\overline{\operatorname{H}}_{ij}]_{i,j = 1,\dots,t}.
\end{align*}

For each $i,j = 1,\dots, t$ we define the vector space isomorphisms
\begin{align*}
    \phi_{ij}: \operatorname{H}_{ij} \; \stackrel{\sim}\longrightarrow \; \langle c \; | \; c:v_i \curvearrowright v_j \in Q \rangle \quad \text{and} \quad
    \overline{\phi}_{ij}: \overline{\operatorname{H}}_{ij} \; \stackrel{\sim}\longrightarrow \; \langle c \; | \; c:v_i \curvearrowright v_j \in \overline{Q} \rangle
\end{align*}
as in Remark \ref{map_def}. If $c:v_i \curvearrowright v_j$ is a path in $Q$, then $c$ does not have the vertex $v$ as an endpoint, and so let $\overline{c}:v_i \curvearrowright v_j$ denote the corresponding path in $\overline{Q}$. Each path in $\overline{Q}$ arises in this way. Thus, for each $i,j = 1,\dots,t$ we can define vector space isomorphisms by specifying their actions on basis elements:
\begin{equation*}
    \Phi_{ij}: \operatorname{H}_{ij} \; \stackrel{\sim}\longrightarrow \; \overline{\operatorname{H}}_{ij}; \quad \phi_{ij}^{-1}(c) \longmapsto \overline{\phi}_{ij}^{-1}(\overline{c}).
\end{equation*}

Clearly, if $c:v_i \curvearrowright v_j$ and $p:v_j \curvearrowright v_k$ are paths in $Q$, then $\overline{cp} = \overline{c} \, \overline{p}$. Using this, along with Proposition \ref{basis_composition}, we have
\begin{align*}
    \Phi_{ij}(\phi_{ik}^{-1}(c) \circ \phi_{kj}^{-1}(p)) 
    &= \Phi_{ij}(\phi_{ij}^{-1}(cp))
    = \overline{\phi}_{ij}^{-1}(\overline{cp})
    = \overline{\phi}_{ij}^{-1}(\overline{c} \,\overline{p})
    = \overline{\phi}_{ik}^{-1}(\overline{c}) \circ \overline{\phi}_{kj}^{-1}(\overline{p})\\
    &= \Phi_{ik}(\phi_{ik}^{-1}(c)) \circ \Phi_{kj}(\phi_{kj}^{-1}(p)).
\end{align*}
By linearity, it follows that for any $f_{ik} \in \operatorname{H}_{ik}$ and $g_{kj} \in \operatorname{H}_{kj}$, $\Phi_{ij}(f_{ik} \circ g_{kj}) = \Phi_{ik}(f_{ik}) \circ \Phi_{kj}(g_{kj})$. From this, and the fact that each $\Phi_{ij}$ is a vector space isomorphism, it is clear that the mapping $\mathbf{\Phi}: A_{Q,\mathbf{d}} \longrightarrow A_{\overline{Q}, \overline{\mathbf{d}}}$ defined by $(\mathbf{\Phi}(f))_{ij} = \Phi_{ij}(f_{ij})$ is an algebra isomorphism.
\end{proof}

\begin{defn}
We say that a quiver $Q$ is \textit{connected} if its underlying graph is path-connected, and \textit{disconnected} otherwise. A \textit{connected component} of $Q$ is a maximal connected subquiver of $Q$.
\end{defn}
Note that the set of connected components of a quiver $Q$ is a partition of $Q$, and that $Q$ is disconnected if and only if it has more than one connected component. Therefore, as the following theorem shows, we can simplify the problem of finding counts for disconnected quivers:
\begin{thm}\label{disconnected quiver}
Let $Q$ be an acyclic quiver with connected components $Q^{1},\dots, Q^{n}$, and let $\mathbf{d} \in \mathbb{Z}_{\geq 0}^{Q_0}$. Then
$[Q, \mathbf{d}] = [Q^1,\mathbf{d}^1]\cdots[Q^n, \mathbf{d}^n]$, where $\mathbf{d}^k = (d_i)_{i \in Q_0^k}$ for $k=1,\dots,n$.
\end{thm}

\begin{proof}
This simply follows from the fact that $A_{Q, \mathbf{d}} \cong A_{Q^1, \mathbf{d}^1} \times \cdots \times A_{Q^n, \mathbf{d}^n}$, as $\mathbb{k}$-algebras.
\end{proof}

\subsection{More Count-Preserving Operations}\label{more reductions}

It can be the case that an operation on a quiver and summand vector can preserve the structure of the radical, while not necessarily preserving the structure of the algebra. We consider such operations in this section. The proofs for these are somewhat more tedious, though similar in nature, to those of Theorems \ref{opposite_quiver} and \ref{zero node removal}. We begin with the following definition:

\begin{defn}[Isomorphic Ideals]\label{isomorphic ideals}
Let $A$,$B$ be $\mathbb{k}$-algebras and $I \unlhd A,J \unlhd B$ be two-sided ideals. We call these ideals isomorphic, writing $I \cong J$, if there exists a bijective $\mathbb{k}$-linear map $\mathbf{\Psi}:I \longrightarrow J$ that also preserves multiplication. That is, so that $\mathbf{\Psi}(i_1 \cdot_A i_2) = \mathbf{\Psi}(i_1) \cdot_B \mathbf{\Psi}(i_2)$, for all $i_1, i_2 \in I$.
\end{defn}
\begin{rem}
It is clear that if $I \cong J$, then $I$ and $J$ have the same number of commuting pairs.
\end{rem}

Through the following theorem and corollary we show that we can change the summand vector's component for a source/sink, so long as we compensate by appropriately duplicating the arrows starting/ending at the source/sink:
\begin{equation*}
    \left[
    \begin{tikzpicture}[baseline={([yshift=-.7ex]current bounding box.center)}]
    \draw (1.48,1) node{$d$};
    \draw (0,0);
    \draw (0,2);
    \draw (3,2);
    \draw (3,0);
    \draw (0.5,1.1) node{\vdots};
    \draw (2.5,1.1) node{\vdots};
    \draw[<-] (0.15,0.15) -- (1.35,0.85);
    \draw[<-] (0.15,1.85) -- (1.35,1.15);
    \draw[->] (1.65,1.15) -- (2.85,1.85);
    \draw[->] (1.65,0.85) -- (2.85,0.15);
    \draw (0.9,0.3);
    \draw (2,0.3);
    \draw (0.9,1.71);
    \draw (2,1.71);
    
    \end{tikzpicture}
    \right]
    \rightsquigarrow
    \left[
    \begin{tikzpicture}[baseline={([yshift=-.7ex]current bounding box.center)}]
    \draw (1.5,1) node{$1$};
    \draw (0,0);
    \draw (0,2);
    \draw (3,2);
    \draw (3,0);
    \draw (0.5,1.1) node{\vdots};
    \draw (2.5,1.1) node{\vdots};
    \draw[<-] (0.15,0.15) -- (1.35,0.85);
    \draw[<-] (0.15,1.85) -- (1.35,1.15);
    \draw[->] (1.65,1.15) -- (2.85,1.85);
    \draw[->] (1.65,0.85) -- (2.85,0.15);
    \draw (0.9,0.3) node{$\#d$};
    \draw (2,0.3) node{$\#d$};
    \draw (0.9,1.71) node{$\#d$};
    \draw (2,1.71) node{$\#d$};
    \end{tikzpicture}
    \right], 
    \quad
    \left[
    \begin{tikzpicture}[baseline={([yshift=-.7ex]current bounding box.center)}]
    \draw (1.48,1) node{$d$};
    \draw (0,0);
    \draw (0,2);
    \draw (3,2);
    \draw (3,0);
    \draw (0.5,1.1) node{\vdots};
    \draw (2.5,1.1) node{\vdots};
    \draw[->] (0.15,0.15) -- (1.35,0.85);
    \draw[->] (0.15,1.85) -- (1.35,1.15);
    \draw[<-] (1.65,1.15) -- (2.85,1.85);
    \draw[<-] (1.65,0.85) -- (2.85,0.15);
    \draw (0.9,0.3);
    \draw (2,0.3);
    \draw (0.9,1.71);
    \draw (2,1.71);
    
    \end{tikzpicture}
    \right]
    \rightsquigarrow
    \left[
    \begin{tikzpicture}[baseline={([yshift=-.7ex]current bounding box.center)}]
    \draw (1.5,1) node{$1$};
    \draw (0,0);
    \draw (0,2);
    \draw (3,2);
    \draw (3,0);
    \draw (0.5,1.1) node{\vdots};
    \draw (2.5,1.1) node{\vdots};
    \draw[->] (0.15,0.15) -- (1.35,0.85);
    \draw[->] (0.15,1.85) -- (1.35,1.15);
    \draw[<-] (1.65,1.15) -- (2.85,1.85);
    \draw[<-] (1.65,0.85) -- (2.85,0.15);
    \draw (0.9,0.3) node{$\#d$};
    \draw (2,0.3) node{$\#d$};
    \draw (0.9,1.71) node{$\#d$};
    \draw (2,1.71) node{$\#d$};
    \end{tikzpicture}
    \right],
\end{equation*}
where ``$\#d$'' denotes that there are $d$ copies of the underlying, drawn arrow (``$\#0$'' means that the underlying, drawn arrow is not in the quiver).

\begin{thm}[Source Conversion]\label{source conversion}
Suppose that $Q$ is acyclic, $v \in Q_0$ is a source, and that $\mathbf{d} \in \mathbb{Z}_{\geq 0}^{Q_0}$ with $d_v = d$. Let $\overline{Q}_0 := Q_0$, $\overline{\mathbf{d}} \in \mathbb{Z}_{\geq 0}^{\overline{Q}_0}$ be such that $\overline{d}_v = 1$ and $\overline{d}_i = d_i$ for all $i \neq v$. Furthermore, let\\ $\mathcal{A} := \{\alpha \in Q_1 \; | \; s(\alpha) = v\}$, $\overline{\mathcal{A}} := \{\alpha^1, \dots, \alpha^d \; | \; \alpha \in \mathcal{A}\}$ consist of $d$-duplicates of each $\alpha \in \mathcal{A}$, and \\$\overline{Q}_1 := Q_1 \cup \overline{\mathcal{A}} \setminus \mathcal{A}$, giving the quiver $\overline{Q} := (\overline{Q}_0, \overline{Q}_1)$. Then $\operatorname{rad}A_{Q,\mathbf{d}} \cong \operatorname{rad}A_{\overline{Q}, \overline{\mathbf{d}}}$. In particular, \\$[Q, \mathbf{d}] = [\overline{Q}, \overline{\mathbf{d}}]$.
\end{thm}
\begin{proof}
The claim is clear if $d=0$. Hence we assume that $d \geq 1$. We write $P_{Q, \mathbf{d}}$ as $\bigoplus_{i = 1}^t P_Q(v_i)$, such that $v_i = v$ for $i = 1, \dots, d$ (and necessarily $v_i \neq v$ for all $i > d$). Since $Q_0 = \overline{Q}_0$ and $d_i = \overline{d}_i$ for all $i \neq v$, we can consistently write $P_{\overline{Q}, \overline{\mathbf{d}}}$ as $\bigoplus_{i = 1,d+1,\dots,t} P_{\overline{Q}}(v_i)$. Therefore
\begin{align*}
    \operatorname{rad}A_{Q,\mathbf{d}} 
    &= [\operatorname{Rad}(P_Q(v_j), P_Q(v_i))]_{i,j = 1,\dots,t} 
    = [\operatorname{R}_{ij}]_{i,j = 1,\dots,t} \quad \text{and}\\
    \operatorname{rad}A_{\overline{Q},\overline{\mathbf{d}}} 
    &= [\operatorname{Rad}(P_{\overline{Q}}(v_j), P_{\overline{Q}}(v_i))]_{i,j = 1,d+1,\dots,t} 
    = [\overline{\operatorname{R}}_{ij}]_{i,j = 1,d+1,\dots,t}.
\end{align*}

For each relevant pair $i,j$, define the vector space isomorphisms
\begin{align*}
    \psi_{ij}: \operatorname{R}_{ij} \; \stackrel{\sim}\longrightarrow \; \langle c \; | \; c:v_i \stackrel{\text{n.c.}}\curvearrowright v_j \in Q \rangle \quad \text{and} \quad
    \overline{\psi}_{ij}: \overline{\operatorname{R}}_{ij} \; \stackrel{\sim}\longrightarrow \; \langle c \; | \; c:v_i \stackrel{\text{n.c.}}\curvearrowright v_j \in \overline{Q} \rangle
\end{align*}
as in Remark \ref{rad_map_def}, giving bases $\{\psi_{ij}^{-1}(c) \; | \; c:v_i \stackrel{\text{n.c.}}\curvearrowright v_j \in Q\}$ of $\operatorname{R}_{ij}$ and $\{\overline{\psi}_{ij}^{-1}(c) \; | \; c:v_i \stackrel{\text{n.c.}}\curvearrowright v_j \in \overline{Q}\}$ of $\overline{\operatorname{R}}_{ij}$. Crucial for later is the observation that, for $v_i, v_j \neq v$ (i.e. for $i,j > d$), $Q$ and $\overline{Q}$ have the same paths $v_i \curvearrowright v_j$ (and thereby non-constant paths $v_i \stackrel{\text{n.c.}}\curvearrowright v_j$).

\begin{itemize}
    \item Understanding $\operatorname{rad}A_{Q, \mathbf{d}}$:
    
    For $j \leq d$, $\operatorname{R}_{ij} = 0$ since there are no non-constant paths $v_i \stackrel{\text{n.c.}}\curvearrowright v$ in $Q$. For $j > d$, we define 
    \begin{equation*}
        \operatorname{R}_{*j} :=
        \begin{bmatrix}
            \operatorname{R}_{1j}\\
            \vdots\\
            \operatorname{R}_{dj}
        \end{bmatrix}
        =
        \operatorname{Rad}(P_Q(v_j), P_Q(v))^d,
    \end{equation*}
    and, for $f_{*j} = [f_{1j}, \dots, f_{dj}]^{\top} \in \operatorname{R}_{*j}$ and $g_{jk} \in \operatorname{R}_{jk}$ with $j,k > d$, define composition as
    \begin{equation}\label{star_composition}
        f_{*j} \circ g_{jk} :=
        \begin{bmatrix}
            f_{1j} \circ g_{jk}\\
            \vdots\\
            f_{dj} \circ g_{jk}
        \end{bmatrix}.
    \end{equation}
    With this we can write
    \begin{equation*}
        \operatorname{rad}A_{Q, \mathbf{d}} =
        \left[
        \begin{array}{c|ccc}
            O_{d\times d} & \operatorname{R}_{*,d+1} & \cdots & \operatorname{R}_{*,t} \\
            \hline
            & \operatorname{R}_{d+1,d+1} & \cdots & \operatorname{R}_{d+1,t}\\
            O_{(t-d) \times d} & \vdots & \ddots & \vdots\\
            & \operatorname{R}_{t,d+1} & \cdots & \operatorname{R}_{t,t}
        \end{array}
        \right]
        \unlhd A_{Q, \mathbf{d}},
    \end{equation*}
    with addition and scalar multiplication in $\operatorname{rad}A_{Q, \mathbf{d}}$ being componentwise, and multiplication being given for $f,g \in \operatorname{rad}A_{Q, \mathbf{d}}$ as
    \begin{align}\label{radA_Q,d composition}
        (f \cdot_{A_{Q, \mathbf{d}}} g)_{*j} = \sum_{k=d+1}^t f_{*k} \circ g_{kj}, \; \text{for $j>d$}, \; \text{and} \;
        (f \cdot_{A_{Q, \mathbf{d}}} g)_{ij} = \sum_{k=d+1}^t f_{ik} \circ g_{kj}, \; \text{for $i,j > d$}.
    \end{align}
    
    \item Understanding $\operatorname{rad}A_{\overline{Q}, \overline{\mathbf{d}}}$:
    
    For $j = 1$ (equivalently here, for $j\leq d$), $\overline{\operatorname{R}}_{ij} = 0$, since there are no non-constant paths $v_i \stackrel{\text{n.c.}}\curvearrowright v$ in $\overline{Q}$. Therefore,
    \begin{equation*}
        \operatorname{rad}A_{\overline{Q}, \overline{\mathbf{d}}} =
        \left[
        \begin{array}{c|ccc}
            0 & \overline{\operatorname{R}}_{1,d+1} & \cdots & \overline{\operatorname{R}}_{1,t}\\
            \hline
            0 & \overline{\operatorname{R}}_{d+1,d+1} & \cdots & \overline{\operatorname{R}}_{d+1,t}\\
            \vdots & \vdots & \ddots & \vdots\\
            0 & \overline{\operatorname{R}}_{t,d+1} & \cdots & \overline{\operatorname{R}}_{t,t}
        \end{array}
        \right]
        \unlhd A_{\overline{Q}, \overline{\mathbf{d}}},
    \end{equation*}
    with addition and scalar multiplication in $\operatorname{rad}A_{\overline{Q}, \overline{\mathbf{d}}}$ being componentwise, and multiplication being given for $\overline{f},\overline{g} \in \operatorname{rad}A_{\overline{Q}, \overline{\mathbf{d}}}$ as
    \begin{equation}\label{radA_Qbar,dbar composition}
        (\overline{f} \cdot_{A_{\overline{Q}, \overline{\mathbf{d}}}} \overline{g})_{ij} = \sum_{k=d+1}^t \overline{f}_{ik} \circ \overline{g}_{kj}, \; \text{for $j > d$ (i.e. for $j > 1$)}.
    \end{equation}
\end{itemize}
Before constructing a mapping between the radicals we define mappings $\Psi_{ij}$ (with domains $\overline{\operatorname{R}}_{ij}$) for each $j > d$ and all $i$ (we need not consider those for $j\leq d$):
\begin{itemize}
    \item For $i=1$ (i.e. $i \leq d$) and $j>d$:
    
    Note that $v_j \neq v$. If $c = \alpha \beta_1 \cdots \beta_r$ is a non-constant path $v\stackrel{\text{n.c.}}\curvearrowright v_j$ in $Q$, then $\alpha \in \mathcal{A}$ and there is a set $\{c^{\gamma}:=\alpha^{\gamma}\beta_1\cdots\beta_r \; | \; \gamma=1,\dots,d\}$ of corresponding non-constant paths $v\stackrel{\text{n.c.}}\curvearrowright v_j$ in $\overline{Q}$, obtained by replacing the first arrow $\alpha$ with its $d$-duplicates in $\overline{\mathcal{A}}$. Furthermore, \\$\{c^{\gamma} \; | \; c:v\stackrel{\text{n.c.}}\curvearrowright v_j \in Q, \gamma = 1,\dots,d\}$ is the complete set of non-constant paths $v\stackrel{\text{n.c.}}\curvearrowright v_j$ in $\overline{Q}$. Hence we can define vector space isomorphisms by specifying their actions on basis elements:
    \begin{align*}
        \Psi_{1j}:\overline{\operatorname{R}}_{1j}\stackrel{\sim}{\longrightarrow}\operatorname{R}_{*j}; \quad
        \overline{\psi}_{1j}^{-1}(c^{\gamma}) \longmapsto
        [0,\dots,0,\psi_{\gamma j}^{-1}(c),0,\dots,0]^{\top}
        =
        [\delta_{\gamma l}\psi_{\gamma j}^{-1}(c)]_{l = 1,\dots,d},
    \end{align*}
    where $\delta_{\gamma l}$ is the Kronecker delta.

    \item For $i,j>d$:
    
    $Q$ and $\overline{Q}$ have the same non-constant paths $v_i \stackrel{\text{n.c.}}\curvearrowright v_j$, so we can once again define vector space isomorphisms by specifying their actions on basis elements:
    \begin{equation*}
        \Psi_{ij}: \overline{\operatorname{R}}_{ij} \stackrel{\sim}{\longrightarrow} \operatorname{R}_{ij}; \quad \overline{\psi}_{ij}^{-1}(c) \longmapsto \psi_{ij}^{-1}(c).
    \end{equation*}
\end{itemize}

We can now define the map $\mathbf{\Psi}:\operatorname{rad}A_{\overline{Q}, \overline{\mathbf{d}}} \longrightarrow \operatorname{rad}A_{Q, \mathbf{d}}$ to be such that
$(\mathbf{\Psi}(\overline{f}))_{*j} = \Psi_{1j}(\overline{f}_{1j})$ 
for $j>d$,
and 
$(\mathbf{\Psi}(\overline{f}))_{ij} = \Psi_{ij}(\overline{f}_{ij})$
for $i,j>d$. $\mathbf{\Psi}$ is a vector space isomorphism, since each $\Psi_{ij}$ is. It remains to show that $\mathbf{\Psi}$ preserves multiplication.

First we show that, for all $j,k>d$ and any $i$, if $c:v_i \stackrel{\text{n.c.}}\curvearrowright v_k$ and $p:v_k \stackrel{\text{n.c.}}\curvearrowright v_j$ are non-constant paths in $\overline{Q}$, then $\Psi_{ij}(\overline{\psi}_{ik}^{-1}(c) \circ \overline{\psi}_{kj}^{-1}(p)) = \Psi_{ik}(\overline{\psi}_{ik}^{-1}(c)) \circ \Psi_{kj}(\overline{\psi}_{kj}^{-1}(p))$:
\begin{itemize}
    \item For $i > d$:
    
    Let $c:v_i \stackrel{\text{n.c.}}\curvearrowright v_k$ and $p:v_k \stackrel{\text{n.c.}}\curvearrowright v_j$ be non-constant paths in $\overline{Q}$. Then
    \begin{align*}
        \Psi_{ij}(\overline{\psi}_{ik}^{-1}(c) \circ \overline{\psi}_{kj}^{-1}(p))
        &\stackrel{\text{Prop. \ref{basis_composition}}}= \Psi_{ij}(\overline{\psi}_{ij}^{-1}(cp))
        = \psi_{ij}^{-1}(cp)
        \stackrel{\text{Prop. \ref{basis_composition}}}= \psi_{ik}^{-1}(c) \circ \psi_{kj}^{-1}(p)\\
        &= \Psi_{ik}(\overline{\psi}_{ik}^{-1}(c)) \circ \Psi_{kj}(\overline{\psi}_{kj}^{-1}(p)).
    \end{align*}
    
    \item For $i = 1$, i.e. $i \leq d$:
    
    Let $c^{\gamma}:v \stackrel{\text{n.c.}}\curvearrowright v_k$ and $p:v_k \stackrel{\text{n.c.}}\curvearrowright v_j$ be non-constant paths in $\overline{Q}$. Then for each $l=1,\dots,d$ we have
    \begin{align*}
        [\Psi_{1j}(\overline{\psi}_{1k}^{-1}(c^{\gamma}) \circ \overline{\psi}_{kj}^{-1}(p))]_l
        &\stackrel{\text{Prop. \ref{basis_composition}}}= [\Psi_{1j}(\overline{\psi}_{1j}^{-1}(c^{\gamma}p))]_l
        = [\Psi_{1j}(\overline{\psi}_{1j}^{-1}((cp)^{\gamma}))]_l
        = \delta_{\gamma l} \psi_{\gamma j}^{-1}(cp)\\
        &\stackrel{\text{Prop. \ref{basis_composition}}}= \delta_{\gamma l} \psi_{\gamma k}^{-1}(c) \circ \psi_{kj}^{-1}(p)
        \stackrel{(\ref{star_composition})}= [\Psi_{1k}(\overline{\psi}_{1k}^{-1}(c^{\gamma})) \circ \Psi_{kj}(\overline{\psi}_{kj}^{-1}(p))]_l.
    \end{align*}
    Therefore,
    \begin{equation*}
        \Psi_{1j}(\overline{\psi}_{1k}^{-1}(c^{\gamma}) \circ \overline{\psi}_{kj}^{-1}(p))
        = \Psi_{1k}(\overline{\psi}_{1k}^{-1}(c^{\gamma})) \circ \Psi_{kj}(\overline{\psi}_{kj}^{-1}(p)).
    \end{equation*}
\end{itemize}
It now follows by linearity that, for all $j,k>d$ and any $i$, for all $\overline{f}_{ik} \in \overline{\operatorname{R}}_{ik}$ and $\overline{g}_{kj} \in \overline{\operatorname{R}}_{kj}$ we have
\begin{equation}\label{distribution}
    \Psi_{ij}(\overline{f}_{ik} \circ \overline{g}_{kj}) = \Psi_{ik}(\overline{f}_{ik}) \circ \Psi_{kj}(\overline{g}_{kj}).
\end{equation}

 Finally, let $\overline{f},\overline{g} \in \operatorname{rad}A_{\overline{Q}, \overline{\mathbf{d}}}$. Then
\begin{itemize}
    \item For $j>d$:
    \begin{align*}
        (\mathbf{\Psi}(\overline{f} \cdot_{A_{\overline{Q}, \overline{\mathbf{d}}}} \overline{g}))_{*j}
        &= \Psi_{1j}((\overline{f} \cdot_{A_{\overline{Q}, \overline{\mathbf{d}}}} \overline{g})_{1j})
        \stackrel{(\ref{radA_Qbar,dbar composition})}= \Psi_{1j}(\sum_{k=d+1}^t \overline{f}_{1k} \circ \overline{g}_{kj})
        \stackrel{(\ref{distribution})}= \sum_{k=d+1}^t \Psi_{1k}(\overline{f}_{1k}) \circ \Psi_{kj}(\overline{g}_{kj})\\
        &= \sum_{k=d+1}^t (\mathbf{\Psi}(\overline{f}))_{*k} \circ (\mathbf{\Psi}(\overline{g}))_{kj}
        \stackrel{(\ref{radA_Q,d composition})}= (\mathbf{\Psi}(\overline{f}) \cdot_{A_{Q, \mathbf{d}}} \mathbf{\Psi}(\overline{g}))_{*j}.
    \end{align*}

    \item Similarly, for $i,j>d$:
    \begin{align*}
        (\mathbf{\Psi}(\overline{f} \cdot_{A_{\overline{Q}, \overline{\mathbf{d}}}} \overline{g}))_{ij}
        &= \Psi_{ij}((\overline{f} \cdot_{A_{\overline{Q}, \overline{\mathbf{d}}}} \overline{g})_{ij})
        \stackrel{(\ref{radA_Qbar,dbar composition})}= \Psi_{ij}(\sum_{k=d+1}^t \overline{f}_{ik} \circ \overline{g}_{kj})
        \stackrel{(\ref{distribution})}= \sum_{k=d+1}^t \Psi_{ik}(\overline{f}_{ik}) \circ \Psi_{kj}(\overline{g}_{kj})\\
        &= \sum_{k=d+1}^t (\mathbf{\Psi}(\overline{f}))_{ik} \circ (\mathbf{\Psi}(\overline{g}))_{kj}
        \stackrel{(\ref{radA_Q,d composition})}= (\mathbf{\Psi}(\overline{f}) \cdot_{A_{Q, \mathbf{d}}} \mathbf{\Psi}(\overline{g}))_{ij}.
    \end{align*}
\end{itemize}
Therefore $\mathbf{\Psi}(\overline{f} \cdot_{A_{\overline{Q}, \overline{\mathbf{d}}}} \overline{g}) = \mathbf{\Psi}(\overline{f}) \cdot_{A_{Q, \mathbf{d}}} \mathbf{\Psi}(\overline{g})$, and the result follows.
\end{proof}

\begin{cor}[Sink Conversion]\label{sink conversion}
Suppose that $Q$ is acyclic, $v \in Q_0$ is a sink, and that $\mathbf{d} \in \mathbb{Z}_{\geq 0}^{Q_0}$ with $d_v = d$. Let $\overline{Q}_0 := Q_0$, $\overline{\mathbf{d}} \in \mathbb{Z}_{\geq 0}^{\overline{Q}_0}$ be such that $\overline{d}_v = 1$ and $\overline{d}_i = d_i$ for all $i \neq v$. Furthermore, let\\ $\mathcal{A} := \{\alpha \in Q_1 \; | \; t(\alpha) = v\}$, $\overline{\mathcal{A}} := \{\alpha^1, \dots, \alpha^d \; | \; \alpha \in \mathcal{A}\}$ consist of $d$-duplicates of each $\alpha \in \mathcal{A}$, and \\$\overline{Q}_1 := Q_1 \cup \overline{\mathcal{A}} \setminus \mathcal{A}$, giving the quiver $\overline{Q} := (\overline{Q}_0, \overline{Q}_1)$. Then $\operatorname{rad}A_{Q,\mathbf{d}} \cong \operatorname{rad}A_{\overline{Q}, \overline{\mathbf{d}}}$. In particular, \\$[Q, \mathbf{d}] = [\overline{Q}, \overline{\mathbf{d}}]$.
\end{cor}
\begin{proof}
Since $v$ is a source in the opposite quiver $Q^{\operatorname{op}}$, we denote by $\overline{Q^{\operatorname{op}}}$ the quiver given as in the statement of Theorem \ref{source conversion}. Note that $\overline{Q^{\operatorname{op}}} = (\overline{Q})^{\operatorname{op}}$ and that $\overline{\mathbf{d}}$ is the same regardless of whether it is understood as in the statement of Theorem \ref{source conversion} or this corollary. Thus,
\begin{align*}
    \operatorname{rad}A_{Q,\mathbf{d}} = (\operatorname{rad}A_{Q, \mathbf{d}}^{\operatorname{op}})^{\operatorname{op}}
    \stackrel{\text{Thm \ref{opposite_quiver}}}\cong (\operatorname{rad}A_{Q^{\operatorname{op}}, \mathbf{d}})^{\operatorname{op}}
    \stackrel{\text{Thm \ref{source conversion}}}\cong (\operatorname{rad}A_{\overline{Q^{\operatorname{op}}}, \overline{\mathbf{d}}})^{\operatorname{op}}\\
    = (\operatorname{rad}A_{(\overline{Q})^{\operatorname{op}}, \overline{\mathbf{d}}})^{\operatorname{op}}
    \stackrel{\text{Thm \ref{opposite_quiver}}}\cong (\operatorname{rad}A_{\overline{Q}, \overline{\mathbf{d}}}^{\operatorname{op}})^{\operatorname{op}}
    = \operatorname{rad}A_{\overline{Q}, \overline{\mathbf{d}}}.
\end{align*}
\end{proof}

\begin{rem} \label{on sink conversion}
We can visualize the argument of the proof of Corollary \ref{sink conversion}, where we apply the operations of arrow reversal (Theorem \ref{opposite_quiver}), source conversion (Theorem \ref{source conversion}), and then arrow reversal again:
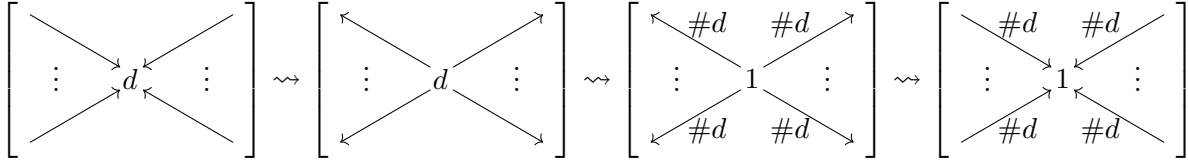
\begin{align*}
    \left[
    \begin{tikzpicture}[baseline={([yshift=-.7ex]current bounding box.center)}]
    \draw (1.48,1) node{$d$};
    \draw (0,0);
    \draw (0,2);
    \draw (3,2);
    \draw (3,0);
    \draw (0.5,1.1) node{\vdots};
    \draw (2.5,1.1) node{\vdots};
    \draw[->] (0.15,0.15) -- (1.35,0.85);
    \draw[->] (0.15,1.85) -- (1.35,1.15);
    \draw[<-] (1.65,1.15) -- (2.85,1.85);
    \draw[<-] (1.65,0.85) -- (2.85,0.15);
    \draw (0.9,0.3);
    \draw (2,0.3);
    \draw (0.9,1.71);
    \draw (2,1.71);
    \end{tikzpicture}
    \right]
    \rightsquigarrow
    \left[
    \begin{tikzpicture}[baseline={([yshift=-.7ex]current bounding box.center)}]
    \draw (1.48,1) node{$d$};
    \draw (0,0);
    \draw (0,2);
    \draw (3,2);
    \draw (3,0);
    \draw (0.5,1.1) node{\vdots};
    \draw (2.5,1.1) node{\vdots};
    \draw[<-] (0.15,0.15) -- (1.35,0.85);
    \draw[<-] (0.15,1.85) -- (1.35,1.15);
    \draw[->] (1.65,1.15) -- (2.85,1.85);
    \draw[->] (1.65,0.85) -- (2.85,0.15);
    \draw (0.9,0.3);
    \draw (2,0.3);
    \draw (0.9,1.71);
    \draw (2,1.71);
    \end{tikzpicture}
    \right]
    \rightsquigarrow
    \left[
    \begin{tikzpicture}[baseline={([yshift=-.7ex]current bounding box.center)}]
    \draw (1.5,1) node{$1$};
    \draw (0,0);
    \draw (0,2);
    \draw (3,2);
    \draw (3,0);
    \draw (0.5,1.1) node{\vdots};
    \draw (2.5,1.1) node{\vdots};
    \draw[<-] (0.15,0.15) -- (1.35,0.85);
    \draw[<-] (0.15,1.85) -- (1.35,1.15);
    \draw[->] (1.65,1.15) -- (2.85,1.85);
    \draw[->] (1.65,0.85) -- (2.85,0.15);
    \draw (0.9,0.3) node{$\#d$};
    \draw (2,0.3) node{$\#d$};
    \draw (0.9,1.71) node{$\#d$};
    \draw (2,1.71) node{$\#d$};
    \end{tikzpicture}
    \right]
    \rightsquigarrow
    \left[
    \begin{tikzpicture}[baseline={([yshift=-.7ex]current bounding box.center)}]
    \draw (1.5,1) node{$1$};
    \draw (0,0);
    \draw (0,2);
    \draw (3,2);
    \draw (3,0);
    \draw (0.5,1.1) node{\vdots};
    \draw (2.5,1.1) node{\vdots};
    \draw[->] (0.15,0.15) -- (1.35,0.85);
    \draw[->] (0.15,1.85) -- (1.35,1.15);
    \draw[<-] (1.65,1.15) -- (2.85,1.85);
    \draw[<-] (1.65,0.85) -- (2.85,0.15);
    \draw (0.9,0.3) node{$\#d$};
    \draw (2,0.3) node{$\#d$};
    \draw (0.9,1.71) node{$\#d$};
    \draw (2,1.71) node{$\#d$};
    \end{tikzpicture}
    \right]
\end{align*}
Of course we must keep in mind that while the radical's structure is unchanged under source conversion, its structure is \textit{opposed} under reversal of all arrows. However, doing so twice (as above) --- or more generally an even number of times --- leaves the radical's structure unchanged, which is what we wanted to show.
\end{rem}

We now introduce what are perhaps our main results, which state that we can split a source/sink, so long as we preserve the arrows starting/ending at that source/sink. We split sources as follows:
\begin{equation}\label{source splitting picture}
    \left[
    \begin{tikzpicture}[baseline={([yshift=-.7ex]current bounding box.center)}]
    \draw (1.48,1) node{$d$};
    \draw (0,0);
    \draw (0,2);
    \draw (3,2);
    \draw (3,0);
    \draw (0.5,1.1) node{\vdots};
    \draw (2.5,1.1) node{\vdots};
    \draw[<-] (0.15,0.15) -- (1.35,0.85);
    \draw[<-] (0.15,1.85) -- (1.35,1.15);
    \draw[->] (1.65,1.15) -- (2.85,1.85);
    \draw[->] (1.65,0.85) -- (2.85,0.15);
    \draw (0.9,0.3);
    \draw (2,0.3);
    \draw (0.9,1.71);
    \draw (2,1.71);
    
    \end{tikzpicture}
    \right]
    \rightsquigarrow
    \left[
    \begin{tikzpicture}[baseline={([yshift=-.7ex]current bounding box.center)}]
    \draw (1.5,1) node{$d$};
    \draw (2.5,1) node{$d$};
    \draw (0,0);
    \draw (0,2);
    \draw (3,2);
    \draw (3,0);
    \draw (0.5,1.1) node{\vdots};
    \draw (3.5,1.1) node{\vdots};
    \draw[<-] (0.15,0.15) -- (1.35,0.85);
    \draw[<-] (0.15,1.85) -- (1.35,1.15);
    \draw[->] (2.65,1.15) -- (3.85,1.85);
    \draw[->] (2.65,0.85) -- (3.85,0.15);
    \draw (0.9,0.3);
    \draw (3,0.3);
    \draw (0.9,1.71);
    \draw (3,1.71);
    \end{tikzpicture}
    \;
    \right]
\end{equation}

\begin{thm}[Source Splitting]\label{source splitting}
Suppose that $Q$ is acyclic, $v\in Q_0$ is a source, and $\mathbf{d} \in \mathbb{Z}_{\geq 0}^{Q_0}$ with $d_v = d$. Let $\{\mathcal{A},\mathcal{B}\}$ be a partition of the set $\{\alpha \in Q_1 \; | \; s(\alpha) = v\}$ of arrows with source vertex $v$. Define the sets $\overline{Q}_0 := Q_0\cup\{v^{\mathcal{A}}, v^{\mathcal{B}}\} \setminus \{v\}$, $\overline{\mathcal{A}} := \{v^{\mathcal{A}} \stackrel{\overline{\alpha}}\longrightarrow t(\alpha) \; | \; \alpha \in \mathcal{A}\}$, $\overline{\mathcal{B}} := \{v^{\mathcal{B}} \stackrel{\overline{\beta}}\longrightarrow t(\beta) \; | \; \beta \in \mathcal{B}\}$, and $\overline{Q}_1 := Q_1\cup(\overline{\mathcal{A}} \cup \overline{\mathcal{B}}) \setminus (\mathcal{A} \cup \mathcal{B})$, so that we have the quiver $\overline{Q} := (\overline{Q}_0, \overline{Q}_1)$. We also define the summand vector $\overline{\mathbf{d}} \in \mathbb{Z}_{\geq 0}^{\overline{Q}_0}$ to be such that $\overline{d}_{v^{\mathcal{A}}}, \overline{d}_{v^{\mathcal{B}}} = d$ and $\overline{d}_i = d_i$ for all $i \in Q_0\setminus \{v\}$. Then $\operatorname{rad}A_{Q,\mathbf{d}} \cong \operatorname{rad}A_{\overline{Q}, \overline{\mathbf{d}}}$. In particular, $[Q, \mathbf{d}] = [\overline{Q}, \overline{\mathbf{d}}]$.
\end{thm}
\begin{proof}
Suppose that $d=1$ (we will prove the claim for general $d$ afterwards). We write $P_{Q, \mathbf{d}}$ as $\bigoplus_{i = 1,2,\dots,t} P_Q(v_i)$, so that $v_1 = v$ (and necessarily $v_i \neq v$ for all $i \geq 2$), and consistently $P_{\overline{Q}, \overline{\mathbf{d}}}$ as $\bigoplus_{i=1^{\mathcal{A}},1^{\mathcal{B}},2,\dots,t}P_{\overline{Q}}(v_i)$, where we defined $v_{1^{\mathcal{A}}} := v^{\mathcal{A}}$ and $v_{1^{\mathcal{B}}} := v^{\mathcal{B}}$. We have
\begin{align*}
    \operatorname{rad}A_{Q,\mathbf{d}} 
    &= [\operatorname{Rad}(P_Q(v_j), P_Q(v_i))]_{i,j = 1,2\dots,t} 
    = [\operatorname{R}_{ij}]_{i,j = 1,2,\dots,t} \quad \text{and}\\
    \operatorname{rad}A_{\overline{Q},\overline{\mathbf{d}}} 
    &= [\operatorname{Rad}(P_{\overline{Q}}(v_j), P_{\overline{Q}}(v_i))]_{i,j = 1^{\mathcal{A}},1^{\mathcal{B}},2,\dots,t} 
    = [\overline{\operatorname{R}}_{ij}]_{i,j = 1^{\mathcal{A}},1^{\mathcal{B}},2,\dots,t}.
\end{align*}
Consider the vector space isomorphisms $\psi_{ij}: \operatorname{R}_{ij} \; \stackrel{\sim}{\longrightarrow} \; \langle c \; | \; c:v_i \stackrel{\text{n.c.}}\curvearrowright v_j \in Q \rangle$ for $i=1,2,\dots,t$, and
$\overline{\psi}_{ij}: \overline{\operatorname{R}}_{ij} \; \stackrel{\sim}{\longrightarrow} \; \langle c \; | \; c:v_i \stackrel{\text{n.c.}}\curvearrowright v_j \in \overline{Q} \rangle$ for $i= 1^{\mathcal{A}},1^{\mathcal{B}},2,\dots,t$, as in Remark \ref{rad_map_def}. These give bases 
$\{\psi_{ij}^{-1}(c) \; | \; c:v_i \stackrel{\text{n.c.}}\curvearrowright v_j \in Q\}$ of $\operatorname{R}_{ij}$ and $\{\overline{\psi}_{ij}^{-1}(c) \; | \; c:v_i \stackrel{\text{n.c.}}\curvearrowright v_j \in \overline{Q}\}$ of $\overline{\operatorname{R}}_{ij}$, for each relevant pair $i,j$.

There is a natural identification between non-constant paths in $Q$ and $\overline{Q}$ (this is clear from (\ref{source splitting picture})). In accordance with this identification, for each $i,j \geq 2$ and non-constant path $c:v_i \stackrel{\text{n.c.}}\curvearrowright v_j$ in $Q$, denote by $\overline{c}$ its corresponding non-constant path $v_i \stackrel{\text{n.c.}}\curvearrowright v_j$ in $\overline{Q}$. For each $j \geq 2$ and non-constant path \\$c=\beta_1\cdots\beta_r:v \stackrel{\text{n.c.}}\curvearrowright v_j$ in $Q$ with first arrow $\beta_1 \in \mathcal{A}$ (respectively $\beta_1 \in \mathcal{B}$), let $\overline{c}$ denote its corresponding non-constant path $v^{\mathcal{A}} \stackrel{\text{n.c.}}\curvearrowright v_j$ (respectively $v^{\mathcal{B}} \stackrel{\text{n.c.}}\curvearrowright v_j$) in $\overline{Q}$.

\begin{itemize}
    \item Understanding $\operatorname{rad}A_{Q, \mathbf{d}}$:
    
    For each $i=1,2,\dots,t$ there are no non-constant paths $c:v_i \stackrel{\text{n.c.}}\curvearrowright v$ in $Q$, so $\operatorname{R}_{ij}=0$ for $j=1$. Therefore
    \begin{align*}
        \operatorname{rad}A_{Q, \mathbf{d}}
        =
        \left[
        \begin{array}{c|ccc}
            0 & \operatorname{R}_{12} & \cdots & \operatorname{R}_{1t}\\
            \hline
            0 & \operatorname{R}_{22} & \cdots & \operatorname{R}_{2t}\\
            \vdots & \vdots & \ddots & \vdots\\
            0 & \operatorname{R}_{t2}& \cdots & \operatorname{R}_{tt}
            \end{array}
        \right]
        \unlhd A_{Q, \mathbf{d}},
    \end{align*}
    with addition and scalar multiplication in $\operatorname{rad}A_{Q, \mathbf{d}}$ being componentwise, and multiplication being given for $f,g \in \operatorname{rad}A_{Q, \mathbf{d}}$ as
    \begin{equation*}
        (f \cdot_{A_{Q, \mathbf{d}}} g)_{ij} = \sum_{k=2}^t f_{ik} \circ g_{kj}, \; \text{for $j \geq 2$}.
    \end{equation*}
    
    \item Understanding $\operatorname{rad}A_{\overline{Q}, \overline{\mathbf{d}}}$:
    
    For each $i=1^{\mathcal{A}},1^{\mathcal{B}},2,\dots,t$ there are no non-constant paths $c:v_i \stackrel{\text{n.c.}}\curvearrowright v^{\mathcal{A}}$ or $c:v_i \stackrel{\text{n.c.}}\curvearrowright v^{\mathcal{B}}$ in $\overline{Q}$, so $\operatorname{\overline{R}}_{ij} = 0$ for $j=1^{\mathcal{A}},1^{\mathcal{B}}$. Furthermore, for each $j \geq 2$, we define the vector space
    $\overline{\operatorname{R}}_{1j} := \overline{\operatorname{R}}_{1^{\mathcal{A}}j} \oplus \overline{\operatorname{R}}_{1^{\mathcal{B}}j}$
    and, for $\overline{f}_{1j} = \left[\overline{f}_{1^{\mathcal{A}}j}, \overline{f}_{1^{\mathcal{B}}j}\right]^{\top} \in \overline{\operatorname{R}}_{1j}$ and $\overline{g}_{jk} \in \overline{\operatorname{R}}_{jk}$ with $j,k \geq 2$, define composition as
    \begin{equation*}
        \overline{f}_{1j} \circ \overline{g}_{jk} :=
        \left[
            \overline{f}_{1^{\mathcal{A}}j} \circ \overline{g}_{jk},
            \overline{f}_{1^{\mathcal{B}}j} \circ \overline{g}_{jk}
        \right]^{\top}.
    \end{equation*}
    With this we can write
    \begin{equation*}
        \operatorname{rad}A_{\overline{Q}, \overline{\mathbf{d}}} =
        \left[
        \begin{array}{c|ccc}
            O_{2\times 2} & \overline{\operatorname{R}}_{12} & \cdots & \overline{\operatorname{R}}_{1t} \\
            \hline
            & \overline{\operatorname{R}}_{22} & \cdots & \overline{\operatorname{R}}_{2t}\\
            O_{(t-1) \times 2} & \vdots & \ddots & \vdots\\
            & \overline{\operatorname{R}}_{t2} & \cdots & \overline{\operatorname{R}}_{tt}
        \end{array}
        \right]
        \unlhd A_{\overline{Q}, \overline{\mathbf{d}}},
    \end{equation*}
    with addition and scalar multiplication in $\operatorname{rad}A_{\overline{Q}, \overline{\mathbf{d}}}$ being componentwise, and multiplication being given for $\overline{f},\overline{g} \in \operatorname{rad}A_{\overline{Q}, \overline{\mathbf{d}}}$ as
    \begin{align*}
        (\overline{f} \cdot_{A_{\overline{Q}, \overline{\mathbf{d}}}} \overline{g})_{ij} = \sum_{k=2}^t \overline{f}_{ik} \circ \overline{g}_{kj}, \; \text{for $i=1,2,\dots,t$ and $j\geq 2$.}
    \end{align*}
\end{itemize}

We now define vector space isomorphisms $\Psi_{ij}:\operatorname{R}_{ij} \stackrel{\sim}{\longrightarrow} \; \overline{\operatorname{R}}_{ij}$ for $i=1,2,\dots,t$ and $j \geq 2$ on bases $\{\psi_{ij}^{-1}(c) \; | \; c:v_i \stackrel{\text{n.c.}}\curvearrowright v_j \in Q\}$ (these are isomorphisms due to the correspondence of non-constant paths in $Q$ and $\overline{Q}$ discussed earlier):
\begin{itemize}
    \item For $j \geq 2$:
    \begin{equation*}
        \Psi_{1j}:\operatorname{R}_{1j} \stackrel{\sim}{\longrightarrow} \; \overline{\operatorname{R}}_{1j};\quad  \psi_{1j}^{-1}(c) \longmapsto
        \begin{cases}
           \left[\overline{\psi}_{1^{\mathcal{A}}j}^{-1}(\overline{c}),0\right]^{\top}, \; \text{if $c=\beta_1\cdots\beta_r$ with $\beta_1 \in \mathcal{A}$, }
           \\
           \left[0,\overline{\psi}_{1^{\mathcal{B}}j}^{-1}(\overline{c})\right]^{\top}, \; \text{if $c=\beta_1\cdots\beta_r$ with $\beta_1 \in \mathcal{B}$.}
        \end{cases}
    \end{equation*}
    
    \item For $i,j \geq 2$:
    \begin{equation*}
        \Psi_{ij}:\operatorname{R}_{ij} \stackrel{\sim}{\longrightarrow} \; \overline{\operatorname{R}}_{ij};\quad  \psi_{ij}^{-1}(c) \longmapsto \overline{\psi}_{ij}^{-1}(\overline{c}).
    \end{equation*}
\end{itemize}

It can easily be checked that, for all $j,k\geq 2$ and any $i=1,2,\dots,t$, if $c:v_i \stackrel{\text{n.c.}}\curvearrowright v_k$ and $p:v_k \stackrel{\text{n.c.}}\curvearrowright v_j$ are non-constant paths in $Q$, then $\Psi_{ij}(\psi_{ik}^{-1}(c) \circ \psi_{kj}^{-1}(p)) = \Psi_{ik}(\psi_{ik}^{-1}(c)) \circ \Psi_{kj}(\psi_{kj}^{-1}(p))$, from which it follows by linearity that $\Psi_{ij}(f_{ik} \circ g_{kj}) = \Psi_{ik}(f_{ik}) \circ \Psi_{kj}(g_{kj})$ for all $f_{ik} \in \operatorname{R}_{ik}$ and $g_{kj} \in \operatorname{R}_{kj}$. Then the map $\mathbf{\Psi}:\operatorname{rad}A_{Q, \mathbf{d}} \longrightarrow \operatorname{rad}A_{\overline{Q}, \overline{\mathbf{d}}}$ defined by $(\mathbf{\Psi}(f))_{ij} = \Psi_{ij}(f_{ij})$ for $i=1,2,\dots,t$ and $j=2,\dots,t$ (which is a vector space isomorphism) can be checked to be multiplication-preserving, analogously to how it was done in the proof of Theorem \ref{source conversion}, and hence $\operatorname{rad}A_{Q, \mathbf{d}} \cong \operatorname{rad}A_{\overline{Q}, \overline{\mathbf{d}}}$ by Definition \ref{isomorphic ideals}. This proves the claim for $d=1$. 

Suppose now that we have general $d \in \mathbb{Z}_{\geq 0}$. By what we have just shown, and repeated applications of Theorem \ref{source conversion},
\begin{align*}
    \left[
    \begin{tikzpicture}[baseline={([yshift=-.7ex]current bounding box.center)}]
    \draw (1.48,1) node{$d$};
    \draw (0,0);
    \draw (0,2);
    \draw (3,2);
    \draw (3,0);
    \draw (0.5,1.1) node{\vdots};
    \draw (2.5,1.1) node{\vdots};
    \draw[<-] (0.15,0.15) -- (1.35,0.85);
    \draw[<-] (0.15,1.85) -- (1.35,1.15);
    \draw[->] (1.65,1.15) -- (2.85,1.85);
    \draw[->] (1.65,0.85) -- (2.85,0.15);
    \draw (0.9,0.3);
    \draw (2,0.3);
    \draw (0.9,1.71);
    \draw (2,1.71);
    \end{tikzpicture}
    \right]
    &\rightsquigarrow
    \left[
    \begin{tikzpicture}[baseline={([yshift=-.7ex]current bounding box.center)}]
    \draw (1.5,1) node{$1$};
    \draw (0,0);
    \draw (0,2);
    \draw (3,2);
    \draw (3,0);
    \draw (0.5,1.1) node{\vdots};
    \draw (2.5,1.1) node{\vdots};
    \draw[<-] (0.15,0.15) -- (1.35,0.85);
    \draw[<-] (0.15,1.85) -- (1.35,1.15);
    \draw[->] (1.65,1.15) -- (2.85,1.85);
    \draw[->] (1.65,0.85) -- (2.85,0.15);
    \draw (0.9,0.3) node{$\#d$};
    \draw (2,0.3) node{$\#d$};
    \draw (0.9,1.71) node{$\#d$};
    \draw (2,1.71) node{$\#d$};
    \end{tikzpicture}
    \right]
    \rightsquigarrow
    \left[
    \begin{tikzpicture}[baseline={([yshift=-.7ex]current bounding box.center)}]
    \draw (1.5,1) node{$1$};
    \draw (2.5,1) node{$1$};
    \draw (0,0);
    \draw (0,2);
    \draw (3,2);
    \draw (3,0);
    \draw (0.5,1.1) node{\vdots};
    \draw (3.5,1.1) node{\vdots};
    \draw[<-] (0.15,0.15) -- (1.35,0.85);
    \draw[<-] (0.15,1.85) -- (1.35,1.15);
    \draw[->] (2.65,1.15) -- (3.85,1.85);
    \draw[->] (2.65,0.85) -- (3.85,0.15);
    \draw (0.9,0.3) node{$\#d$};
    \draw (3,0.3) node{$\#d$};
    \draw (0.9,1.71) node{$\#d$};
    \draw (3,1.71) node{$\#d$};
    \end{tikzpicture}\;
    \right]\\
    &\rightsquigarrow
    \left[
    \begin{tikzpicture}[baseline={([yshift=-.7ex]current bounding box.center)}]
    \draw (1.5,1) node{$d$};
    \draw (2.5,1) node{$1$};
    \draw (0,0);
    \draw (0,2);
    \draw (3,2);
    \draw (3,0);
    \draw (0.5,1.1) node{\vdots};
    \draw (3.5,1.1) node{\vdots};
    \draw[<-] (0.15,0.15) -- (1.35,0.85);
    \draw[<-] (0.15,1.85) -- (1.35,1.15);
    \draw[->] (2.65,1.15) -- (3.85,1.85);
    \draw[->] (2.65,0.85) -- (3.85,0.15);
    \draw (3,0.3) node{$\#d$};
    \draw (3,1.71) node{$\#d$};
    \end{tikzpicture}\;
    \right]
    \rightsquigarrow
    \left[
    \begin{tikzpicture}[baseline={([yshift=-.7ex]current bounding box.center)}]
    \draw (1.5,1) node{$d$};
    \draw (2.5,1) node{$d$};
    \draw (0,0);
    \draw (0,2);
    \draw (3,2);
    \draw (3,0);
    \draw (0.5,1.1) node{\vdots};
    \draw (3.5,1.1) node{\vdots};
    \draw[<-] (0.15,0.15) -- (1.35,0.85);
    \draw[<-] (0.15,1.85) -- (1.35,1.15);
    \draw[->] (2.65,1.15) -- (3.85,1.85);
    \draw[->] (2.65,0.85) -- (3.85,0.15);
    \end{tikzpicture}\;
    \right]
\end{align*}
is a sequence of radical-preserving operations.
\end{proof}

As the following corollary shows, we have the dual notion of splitting sinks as follows:
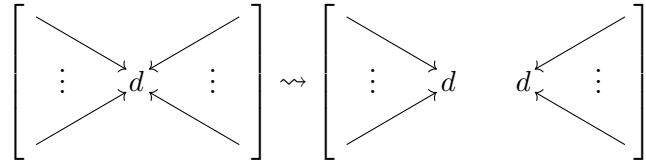
\begin{equation*}
    \left[
    \begin{tikzpicture}[baseline={([yshift=-.7ex]current bounding box.center)}]
    \draw (1.48,1) node{$d$};
    \draw (0,0);
    \draw (0,2);
    \draw (3,2);
    \draw (3,0);
    \draw (0.5,1.1) node{\vdots};
    \draw (2.5,1.1) node{\vdots};
    \draw[->] (0.15,0.15) -- (1.35,0.85);
    \draw[->] (0.15,1.85) -- (1.35,1.15);
    \draw[<-] (1.65,1.15) -- (2.85,1.85);
    \draw[<-] (1.65,0.85) -- (2.85,0.15);
    \draw (0.9,0.3);
    \draw (2,0.3);
    \draw (0.9,1.71);
    \draw (2,1.71);
    
    \end{tikzpicture}
    \right]
    \rightsquigarrow
    \left[
    \begin{tikzpicture}[baseline={([yshift=-.7ex]current bounding box.center)}]
    \draw (1.5,1) node{$d$};
    \draw (2.5,1) node{$d$};
    \draw (0,0);
    \draw (0,2);
    \draw (3,2);
    \draw (3,0);
    \draw (0.5,1.1) node{\vdots};
    \draw (3.5,1.1) node{\vdots};
    \draw[->] (0.15,0.15) -- (1.35,0.85);
    \draw[->] (0.15,1.85) -- (1.35,1.15);
    \draw[<-] (2.65,1.15) -- (3.85,1.85);
    \draw[<-] (2.65,0.85) -- (3.85,0.15);
    \draw (0.9,0.3);
    \draw (3,0.3);
    \draw (0.9,1.71);
    \draw (3,1.71);
    \end{tikzpicture}\;
    \right]
\end{equation*}

\begin{cor}[Sink Splitting]\label{sink splitting}
Suppose that $Q$ is acyclic, $v\in Q_0$ is a sink, and $\mathbf{d} \in \mathbb{Z}_{\geq 0}^{Q_0}$ with $d_v = d$. Let $\{\mathcal{A},\mathcal{B}\}$ be a partition of the set $\{\alpha \in Q_1 \; | \; t(\alpha) = v\}$ of arrows with terminal vertex $v$. Define the sets $\overline{Q}_0 := Q_0\cup\{v^{\mathcal{A}}, v^{\mathcal{B}}\} \setminus \{v\}$, $\overline{\mathcal{A}} := \{s(\alpha) \stackrel{\overline{\alpha}}\longrightarrow v^{\mathcal{A}} \; | \; \alpha \in \mathcal{A}\}$, $\overline{\mathcal{B}} := \{s(\beta) \stackrel{\overline{\beta}}\longrightarrow v^{\mathcal{B}} \; | \; \beta \in \mathcal{B}\}$, and $\overline{Q}_1 := Q_1\cup(\overline{\mathcal{A}} \cup \overline{\mathcal{B}}) \setminus (\mathcal{A} \cup \mathcal{B})$, so that we have the quiver $\overline{Q} := (\overline{Q}_0, \overline{Q}_1)$. We also define the summand vector $\overline{\mathbf{d}} \in \mathbb{Z}_{\geq 0}^{\overline{Q}_0}$ to be such that $\overline{d}_{v^{\mathcal{A}}}, \overline{d}_{v^{\mathcal{B}}} = d$ and $\overline{d}_i = d_i$ for all $i \in Q_0\setminus \{v\}$. Then $\operatorname{rad}A_{Q,\mathbf{d}} \cong \operatorname{rad}A_{\overline{Q}, \overline{\mathbf{d}}}$. In particular, $[Q, \mathbf{d}] = [\overline{Q}, \overline{\mathbf{d}}]$.
\end{cor}

\begin{proof}
Analogous to the argument of Remark \ref{on sink conversion}, we have that the sequence of operations
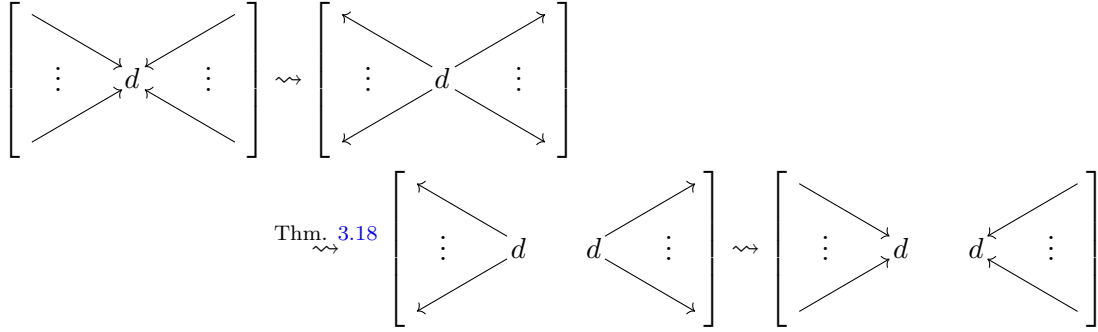
\begin{align*}
    \left[
    \begin{tikzpicture}[baseline={([yshift=-.7ex]current bounding box.center)}]
    \draw (1.48,1) node{$d$};
    \draw (0,0);
    \draw (0,2);
    \draw (3,2);
    \draw (3,0);
    \draw (0.5,1.1) node{\vdots};
    \draw (2.5,1.1) node{\vdots};
    \draw[->] (0.15,0.15) -- (1.35,0.85);
    \draw[->] (0.15,1.85) -- (1.35,1.15);
    \draw[<-] (1.65,1.15) -- (2.85,1.85);
    \draw[<-] (1.65,0.85) -- (2.85,0.15);
    \draw (0.9,0.3);
    \draw (2,0.3);
    \draw (0.9,1.71);
    \draw (2,1.71);
    \end{tikzpicture}
    \right]
    &\rightsquigarrow
    \left[
    \begin{tikzpicture}[baseline={([yshift=-.7ex]current bounding box.center)}]
    \draw (1.48,1) node{$d$};
    \draw (0,0);
    \draw (0,2);
    \draw (3,2);
    \draw (3,0);
    \draw (0.5,1.1) node{\vdots};
    \draw (2.5,1.1) node{\vdots};
    \draw[<-] (0.15,0.15) -- (1.35,0.85);
    \draw[<-] (0.15,1.85) -- (1.35,1.15);
    \draw[->] (1.65,1.15) -- (2.85,1.85);
    \draw[->] (1.65,0.85) -- (2.85,0.15);
    \draw (0.9,0.3);
    \draw (2,0.3);
    \draw (0.9,1.71);
    \draw (2,1.71);
    \end{tikzpicture}
    \right]\\
    &\stackrel{\text{Thm. \ref{source splitting}}}\rightsquigarrow
    \left[
    \begin{tikzpicture}[baseline={([yshift=-.7ex]current bounding box.center)}]
    \draw (1.5,1) node{$d$};
    \draw (2.5,1) node{$d$};
    \draw (0,0);
    \draw (0,2);
    \draw (3,2);
    \draw (3,0);
    \draw (0.5,1.1) node{\vdots};
    \draw (3.5,1.1) node{\vdots};
    \draw[<-] (0.15,0.15) -- (1.35,0.85);
    \draw[<-] (0.15,1.85) -- (1.35,1.15);
    \draw[->] (2.65,1.15) -- (3.85,1.85);
    \draw[->] (2.65,0.85) -- (3.85,0.15);
    \end{tikzpicture}\;
    \right]
    \rightsquigarrow
    \left[
    \begin{tikzpicture}[baseline={([yshift=-.7ex]current bounding box.center)}]
    \draw (1.5,1) node{$d$};
    \draw (2.5,1) node{$d$};
    \draw (0,0);
    \draw (0,2);
    \draw (3,2);
    \draw (3,0);
    \draw (0.5,1.1) node{\vdots};
    \draw (3.5,1.1) node{\vdots};
    \draw[->] (0.15,0.15) -- (1.35,0.85);
    \draw[->] (0.15,1.85) -- (1.35,1.15);
    \draw[<-] (2.65,1.15) -- (3.85,1.85);
    \draw[<-] (2.65,0.85) -- (3.85,0.15);
    \end{tikzpicture}\;
    \right]
\end{align*}
leaves the radical's structure unchanged.
\end{proof}

The significance of Theorem \ref{source splitting} and Corollary \ref{sink splitting} is that when a source/sink of a connected quiver is the only point of contact between otherwise distinct connected components, then we can split the quiver at the source/sink so as to disconnect it, yielding these connected components while preserving the count. By Theorem \ref{disconnected quiver}, the count of the original quiver is then the product of the counts of the connected components.

\section{First Cases} \label{examples}
The results and operations introduced in \S \ref{reductions} and \S \ref{more reductions} are particularly effective for understanding the counts of quivers and summand vectors where after removal of zero-vertices as in Theorem \ref{zero node removal} there are no remaining paths of length greater than $2$. For such quivers and summand vectors we see that it suffices to understand counts of the forms
\begin{align*}
    \left[
    \begin{tikzpicture}[baseline={([yshift=-.7ex]current bounding box.center)}]
        \draw (0,0) node{1};
    \end{tikzpicture}
    \right],
    \left[
    \begin{tikzpicture}[baseline={([yshift=-.7ex]current bounding box.center)}]
        \draw (0,0) node{1};
        \draw (1.5,0) node{1};
        \draw[->] (0.2,0) -- (1.3,0);
    \end{tikzpicture}
    \right], \; \text{and} \;
    \left[
    \begin{tikzpicture}[baseline={([yshift=-.7ex]current bounding box.center)}]
        \draw (0,0) node{$l$};
        \draw (1.5,0) node{$d$};
        \draw (3,0) node{$m$};
        \draw[->] (0.2,0) -- (1.3,0);
        \draw[<-] (2.8,0) -- (1.7,0);
    \end{tikzpicture}
    \right].
\end{align*}
It is clear that 
$\left[
\begin{tikzpicture}[baseline={([yshift=-.7ex]current bounding box.center)}]
    \draw (0,0) node{1};
\end{tikzpicture}
\right] = 1$, and that 
$\left[
\begin{tikzpicture}[baseline={([yshift=-.7ex]current bounding box.center)}]
    \draw (0,0) node{1};
    \draw (1.5,0) node{1};
    \draw[->] (0.2,0) -- (1.3,0);
\end{tikzpicture}
\right] = q^2$.

\begin{prop} For $l,d,m \geq 1$ we have
\begin{equation*}
\left[
\begin{tikzpicture}[baseline={([yshift=-.7ex]current bounding box.center)}]
    \draw (0,0) node{$l$};
    \draw (1.5,0) node{$d$};
    \draw (3,0) node{$m$};
    \draw[->] (0.2,0) -- (1.3,0);
    \draw[<-] (2.8,0) -- (1.7,0);
\end{tikzpicture}
\right]
=
q^{2lm} \sum_{i=\max\{0,2d-l\}}^{2d}
\left(
q^{mi}
\prod_{j=0}^{i-1}\left( \frac{q^{2d}-q^j}{q^i-q^j}\right)
\prod_{j=0}^{2d-i-1}\left(q^l-q^j\right)
\right).
\end{equation*}
\end{prop}
\begin{proof}
Keeping in mind the underlying quiver, we see that $\operatorname{rad}A_{Q,\mathbf{d}}$ is isomorphic to the set of block matrices of the form
\begin{equation*}
\left[
\begin{array}{c|c|c}
     O & A & B \\\hline
     O & O & C \\\hline
     O & O & O
\end{array}
\right],
\end{equation*} 
where $A \in \operatorname{Mat}_{l \times d}(\mathbb{k})$, $B \in \operatorname{Mat}_{l \times m}(\mathbb{k})$, and $C \in \operatorname{Mat}_{d \times m}(\mathbb{k})$. A pair
\begin{equation}\label{form of pair}
\left(
\left[
\begin{array}{c|c|c}
     O & A & B \\ \hline
     O & O & C \\ \hline
     O & O & O
\end{array}
\right],
\left[
\begin{array}{c|c|c}
     O & \overline{A} & \overline{B} \\\hline
     O & O & \overline{C} \\\hline
     O & O & O
\end{array}
\right]
\right)
\end{equation}
commutes if and only if $A \overline{C} - \overline{A} C = O$. Writing this condition as 
$
\left[ A \middle| -\overline{A} \right]_{l \times 2d}
\left[
\begin{array}{c}
\overline{C}\\
\hline
C
\end{array}
\right]_{2d \times m} = O
$ and keeping in mind $B$ and $\overline{B}$, we see that the number of commuting pairs as in (\ref{form of pair}), the count, is $q^{2lm}$ multiplied by the number of pairs $(X,Y)$ where $X$ is a linear mapping $\mathbb{k}^{2d} \longrightarrow \mathbb{k}^{l}$ and $Y$ is a linear mapping $\mathbb{k}^{m} \longrightarrow \mathbb{k}^{2d}$ with $\operatorname{im}Y \subseteq \operatorname{ker}X$. Writing all this concisely, we have 
\begin{equation*}
\left[
\begin{tikzpicture}[baseline={([yshift=-.7ex]current bounding box.center)}]
    \draw (0,0) node{$l$};
    \draw (1.5,0) node{$d$};
    \draw (3,0) node{$m$};
    \draw[->] (0.2,0) -- (1.3,0);
    \draw[<-] (2.8,0) -- (1.7,0);
\end{tikzpicture}
\right]
=
q^{2lm}
\left|
\left\{
\left(X:\mathbb{k}^{2d} \longrightarrow \mathbb{k}^{l},Y:\mathbb{k}^{m} \longrightarrow \mathbb{k}^{2d}\right)
\; \middle| \;
\operatorname{im}Y \subseteq \operatorname{ker}X
\right\}
\right|,
\end{equation*}
where the fact that $X$ and $Y$ are linear is understood. Next, note that
\begin{align*}
\left|
\left\{
\left(X:\mathbb{k}^{2d} \longrightarrow \mathbb{k}^{l},Y:\mathbb{k}^{m} \longrightarrow \mathbb{k}^{2d} \right)
\; \middle| \;
\operatorname{im}Y \subseteq \operatorname{ker}X
\right\}
\right|
 =
\left|
\left\{
\left(X:\mathbb{k}^{2d} \longrightarrow \mathbb{k}^{l},Y:\mathbb{k}^{m} \longrightarrow \operatorname{ker}X\right)
\right\}
\right|
\\=
\left|
\left\{
\left(X:\mathbb{k}^{2d} \longrightarrow \mathbb{k}^{l},Y:\mathbb{k}^{m} \longrightarrow \mathbb{k}^{\operatorname{dim}(\operatorname{ker}X)}\right)
\right\}
\right|
\\=
\sum_{i=\max\{0,2d - l\}}^{2d}
\left|
\left\{
X:\mathbb{k}^{2d} \longrightarrow \mathbb{k}^{l}
\; \middle| \;
\operatorname{dim}(\operatorname{ker}X) = i
\right\}
\right|
\cdot
\underbrace{
\left|
\left\{
Y:\mathbb{k}^{m} \longrightarrow \mathbb{k}^i
\right\}
\right|}_{q^{mi}}.
\end{align*}
We now determine 
$\left|
\left\{
X:\mathbb{k}^{2d} \longrightarrow \mathbb{k}^{l}
\; \middle| \;
\operatorname{dim}(\operatorname{ker}X) = i
\right\}
\right|$. 
By universality, a linear map $X:\mathbb{k}^{2d} \longrightarrow \mathbb{k}^{l}$ factors \textit{uniquely} through the canonical mapping $\pi:\mathbb{k}^{2d} \longrightarrow \mathbb{k}^{2d}/\operatorname{ker}X$, inducing an isomorphism \\$\overline{X}:\mathbb{k}^{2d}/\operatorname{ker}X \stackrel{\sim}\longrightarrow \operatorname{im}X$:
\begin{equation*}
\begin{tikzcd}[row sep=10ex, column sep=10ex]
    \mathbb{k}^{2d} \arrow[r, "\pi"] \arrow[dr, "X", swap] 
    & \mathbb{k}^{2d}/\operatorname{ker}X \arrow[d, dashed, "\sim"  labl,"\exists ! \overline{X}"] \\
    & \operatorname{im}X
\end{tikzcd}.
\end{equation*}
Thus we have a bijection from the set of linear mappings $\mathbb{k}^{2d} \longrightarrow \mathbb{k}^{l}$ with kernel of dimension $i$ to the set of tuples $(U,V,\phi)$ where $U \subseteq \mathbb{k}^{2d}$ and $V \subseteq \mathbb{k}^{l}$ are subspaces of dimensions $i$ and $2d-i$ respectively and $\phi$ is an isomorphism $\mathbb{k}^{2d}/U \stackrel{\sim}\longrightarrow V$. This bijection sends $X \longmapsto (\operatorname{ker}X, \operatorname{im}X, \overline{X})$. Due to this bijection, we see that
$
\left|
\left\{
X:\mathbb{k}^{2d} \longrightarrow \mathbb{k}^{l}
\; \middle| \;
\operatorname{dim}(\operatorname{ker}X) = i
\right\}
\right|
$
is equal to
\begin{align*}
\left|
\left\{
U \subseteq \mathbb{k}^{2d} 
\; \middle| \;
\dim U = i
\right\}
\right|
\cdot
\left|
\left\{
V \subseteq \mathbb{k}^{l} 
\; \middle| \;
\dim V = 2d-i
\right\}
\right|
\cdot
\left|
GL_{2d-i}(\mathbb{k})
\right|,
\end{align*}
where the fact that $U$ and $V$ are subspaces is understood. The total number of ordered bases for subspaces of dimension $i$ in $\mathbb{k}^{2d}$ is $\prod_{j=0}^{i-1}(q^{2d}-q^j)$ (the empty product is understood to be 1 for the case $i = 0$), while the number of ordered bases for a given subspace $U \subseteq \mathbb{k}^{2d}$ of dimension $i$ is $\prod_{j=0}^{i-1}(q^{i}-q^j)$. Hence 
\begin{align*}
\left|
\left\{
U \subseteq \mathbb{k}^{2d} 
\; \middle| \;
\dim U = i
\right\}
\right|
&=
\prod_{j=0}^{i-1}\left(\frac{q^{2d}-q^j}{q^{i}-q^j}\right), 
\; \text{and similarly}\\
\left|
\left\{
V \subseteq \mathbb{k}^{l} 
\; \middle| \;
\dim V = 2d-i
\right\}
\right|
&=
\prod_{j=0}^{2d-i-1}\left(\frac{q^{l}-q^j}{q^{2d-i}-q^j}\right).
\end{align*}
Finally, we have $\left|GL_{2d-i}(\mathbb{k})\right| = \prod_{j=0}^{2d-i-1}(q^{2d-i}-q^j)$, yielding
\begin{equation*}
\left|
\left\{
X:\mathbb{k}^{2d} \longrightarrow \mathbb{k}^{l}
\; \middle| \;
\operatorname{dim}(\operatorname{ker}X) = i
\right\}
\right|
=
\prod_{j=0}^{i-1}\left(\frac{q^{2d}-q^j}{q^{i}-q^j}\right)
\prod_{j=0}^{2d-i-1}\left(q^{l}-q^j\right),
\end{equation*}
and the result follows. 
\end{proof}

\section{Further Investigations}
\label{furtherdirections}
In this section we present a few questions related to counts that are of interest to attack the quiver version of Higman's conjecture stated in Section \ref{a quiver generalization} and to understand its link to Hall algebras.

\subsection{Commuting endomorphisms}
As mentioned in the introduction, determining the character of the Hall algebra of pairs $(P,f)$ of a projective representation $P$ of an acyclic quiver together with an endomorphism $f\in\textrm{rad}(\operatorname{End}(P))$, which has links with different Hall algebras considered in the literature \cite{ruan2021lie,bridgeland2013quantum}, involves the number of commuting pairs in $\mathrm{Aut}(P)\times\mathrm{rad}(\mathrm{End}(P))$. As in \cite[Theorem 1.1]{Mozgovoy}, this number is related to the number of commuting pairs in $\mathrm{End}(P)\times\mathrm{rad}(\mathrm{End}(P))$ and the polynomiality of one implies the polynomiality of the other. As in Definition \ref{dfn:projectiverepresentation} and Remark \ref{rigorous sense}, a projective representation of an acyclic quiver $Q$ is determined by its summand vector $\mathbf{d}\in\mathbb{Z}_{\geq 0}^{Q_0}$. We let $\overline{[Q,\mathbf{d}]}$ denote the number of commuting pairs in $A_{Q,\mathbf{d}}\times\operatorname{rad}A_{Q,\mathbf{d}}$.

\begin{conj}
 Let $Q$ be an acyclic quiver and $\mathbf{d}\in \mathbb{Z}_{\geq 0}^{Q_0}$. The quantity $\overline{[Q,\mathbf{e}]}$ is a polynomial in $q$ for all $\mathbf{e}\leq\mathbf{d}$ if and only if $[Q,\mathbf{e}]$ is a polynomial in $q$ for all $\mathbf{e}\leq\mathbf{d}$. Moreover, $\overline{[Q,\mathbf{d}]}$ can be expressed in terms of $[Q,\mathbf{e}]$ for $\mathbf{e}\leq \mathbf{d}$. 
\end{conj}

This conjecture is the analog of \cite[Corollary 4.2.5]{evseev2007groups}, which deals with the classical case (that is commuting pairs in $B_n\times T_n$).

Regarding the examples computed in \cite[page 93]{evseev2007groups} and the various positivity conjectures appearing in the theory of quiver representations (Kac conjectured the positivity of the polynomials counting representations of quivers solved in \cite{hausel2013positivity} and also in \cite{davison2018purity}, the positivity of absolutely cuspidal polynomials conjectured in \cite{bozec2019counting} and the deep geometric methods developed to tackle them), we take the risk of conjecturing the following

\begin{conj}
 For any acyclic quiver $Q$ and $\mathbf{d}\in\mathbb{Z}_{\geq 0}^{Q_0}$, the quantity $\overline{[Q,\mathbf{d}]}$ is a polynomial in $q$ with nonnegative coefficients.
\end{conj}

\subsection{Reduction of the conjecture}
Using the operations described in Theorem \ref{source conversion}, Corollary \ref{sink conversion}, Theorem \ref{source splitting} and Corollary \ref{sink splitting}, Conjecture \ref{THE CONJECTURE} is equivalent to the same conjecture for oriented quivers with a single source and a single sink (that is, quivers with a set of vertices $\{1,\hdots,n\}$ for some $n$ with only source $1$, only sink $n$, and the condition that there exists an arrow $i\rightarrow j$ only if $i<j$).

\begin{prop}
Let $Q$ be an oriented quiver with vertices $1,\dots,n$ and suppose that $n$ is the only sink. Let $\mathbf{d}\in\mathbb{Z}_{\geq 0}^{Q_0}$ be a summand vector, giving the projective representation $P_{Q,\mathbf{d}}$ of $Q$. We denote by $V=\bigoplus_{i=1}^nV_i$ the underlying vector space of $P_{Q,\mathbf{d}}$. Then the map
\begin{equation*}
\begin{matrix}
    A_{Q,\mathbf{d}} = \operatorname{End}(P_{Q,\mathbf{d}}) & \longrightarrow & \operatorname{End}(V_n);\\
    f = (f_1,\hdots,f_n) & \longmapsto & f_n
\end{matrix}
\end{equation*}
is injective.
\end{prop}
\begin{proof}
Let $1\leq i\leq n$. Since $Q$ is acyclic, there are a finite number of paths starting at $i$ and hence there exists a path $c$ with $s(c) = i$ of maximal length. That is, so that $l(c) \geq l(p)$ for all paths $p$ with $s(p)=i$. Such a path must end at a sink, so necessarily $t(c) = n$. The linear map $\varphi_c$:$V_i\longrightarrow V_n$ given by the action of $c$ is injective, and so the equality $\varphi_c f_i=f_n\varphi_c$ implies that $f_i$ is determined by $f_n$, proving the injectivity of the map $A_{Q,\mathbf{d}} \longrightarrow \operatorname{End}(V_n)$.
\end{proof}

In view of this result, it seems interesting to understand the image of this map as well as the image of the radical $\operatorname{rad}A_{Q,\mathbf{d}}$. It is reasonable to expect a description in terms of patterns groups. Which patterns groups can be obtained this way is an open question of interest.


\bibliographystyle{alpha}
\bibliography{Bibliography}

\end{document}